\documentclass[12pt]{amsart}
\usepackage{amsmath}
\usepackage{amssymb}
\usepackage{amsthm}
\usepackage{amsfonts}
\usepackage{color}
\usepackage{hyperref}
\usepackage[margin=1.2in]{geometry}

\newtheorem{theorem}{Theorem}[section]
\newtheorem{lemma}[theorem]{Lemma}
\newtheorem{corollary}[theorem]{Corollary}
\newtheorem{proposition}[theorem]{Proposition}

\theoremstyle{definition}
\newtheorem{definition}[theorem]{Definition}

\theoremstyle{remark}
\newtheorem{remark}[theorem]{Remark}

\numberwithin{equation}{section}

\newcommand{\q}{\\&\quad}
\newcommand{\db}{\bar{\partial}}
\newcommand{\dd}{\partial}
\newcommand{\ov}{\overline}
\newcommand{\C}{\Omega}
\newcommand{\me}{\omega}
\newcommand{\p}{\varphi}
\newcommand{\n}{\nabla}
\newcommand{\nb}{\ov{\nabla}}
\newcommand{\dt}{\dd_{t}}
\newcommand{\D}{\Delta}
\newcommand{\Db}{\ov{\Delta}}
\newcommand{\heat}{(\dt-\D)}
\newcommand{\Ga}{\Gamma}
\newcommand{\tr}{\text{\rm tr}_{g}}

\newcommand{\ci}{\circ}
\newcommand{\Tb}{\overline{T}}
\newcommand{\ga}{\gamma}
\newcommand{\Ric}{\text{\rm Ric}}

\title[Derivative Estimates of Pluriclosed Flow]{Derivative Estimates of Pluriclosed Flow}
\author{Yanan Ye}
\address{}
\email{\href{mailto:yeyanan@outlook.com}{yeyanan@outlook.com}}
\date{}

\begin{document}
\maketitle

\begin{abstract}
We provide a derivative estimate for the pluriclosed flow, controlling higher order derivatives of Chern curvature and torsion using the Chern curvature.
Moreover, we derive an estimate for torsion tensor using Chern Ricci curvature in dimension two.
And in the Hermitian-symplectic case, we find a monotonic quantity and use it to prove that all Hermitian-symplectic solitons are K{\"a}hler Ricci solitons.
\end{abstract}

\tableofcontents

\section{Introduction}
The pluriclosed flow is a parabolic flow of Hermitian metrics that was introduced by Streets and Tian in \cite{MR2673720}.
Motivated by Perelman\rq{}s celebrated work on the Ricci flow \cite{2002math.....11159P,2003math......7245P,2003math......3109P}, they considered to classify complex manifolds, particularly the Enriques-Kodaira\rq{}s VII surfaces (e.g. see \cite{MR2030225}), using a geometric flow \cite{MR4181011}.
Using the Bismut Ricci form $\rho$, the pluriclosed flow is formulated as
\begin{align*}
\dt\me=-\rho^{1,1}(\me).
\end{align*}

In the context of the Ricci flow, higher order derivatives are controlled by the curvature tensor.
Compactness results are obtained when the curvature is bounded.
However, in complex geometry, both curvature and torsion of the Chern connection need to be considered.
Streets and Tian established a derivative estimate for the pluriclosed flow, controlling higher order derivatives of curvature and torsion using the Chern curvature $\C$, torsion $T$ and gradient of torsion $DT$ (as explored in \cite{MR2673720,MR2781927}).
Additionally, they proved that in dimension 2, if the Chern curvature is bounded, the pluriclosed flow can exist for a long time (see \cite{MR2673720}).
This paper generalizes their results and demonstrates that higher derivatives of curvature and torsion can be solely controlled by the Chern curvature.
More precisely, the following theorem is established.

\begin{theorem}\label{thm-BBS-curvature}
For a pluriclosed flow $g(t)$ on a compact complex manifold $(M^{2n},J)$, given positive constants $\tau>0$, $K\geq 1$ and integer $m\geq 1$, there exists a constant $C_{m}$ depending only on the dimension $n$ and $\max\{\tau,1\}$ such that if
\begin{align*}
|\C(x,t)|_{g(x,t)}\leq K
,\quad
x\in M  \ and\   t\in [0,\frac{\tau}{K^2}]
\end{align*}
then
\begin{align*}
|D^m\C(x,t)|_{g(x,t)}\leq \frac{C_m K}{t^{(m+1)/2}}
,\quad
|D^{m}T(x,t)|_{g(x,t)}\leq \frac{C_{m}K}{t^{m/2}}
,\quad
x\in M \ and\ t\in [0,\frac{\tau}{K^2}].
\end{align*}
\end{theorem}
As an application, we obtain
\begin{theorem}\label{thm-maximalExistenceTime}
If the maximal existence time $\tau$ of pluriclosed flow $g(t)$ is finite, then we have
\begin{align*}
\lim_{t\to\tau-}\left(
	\max_{x\in M}|\C(x,t)|_{g(x,t)}
	\right)=+\infty.
\end{align*}
\end{theorem}
This theorem generalizes Streets and Tian\rq{}s result in dimension 2.
Several regularity results of the pluriclosed flow described by the Bismut connection can be found in \cite{2021arXiv210613716G,MR2755684}.
Moreover, we can obtain a torsion estimate in dimension 2.
%\begin{theorem}\label{thm-HSdim2-T2}
%For a compact complex surface $(M^4,J,g(t))$ that admits a pluriclosed flow with Hermitian-symplectic initial data on $[0,\tau)$ (where $\tau\leq+\infty$), if
%\begin{align*}
%|(S+\Ric)(x,t)|_{g(x,t)}\leq K
%,\quad 
%x\in M \ and\ t\in[0,\tau)
%\end{align*}
%then there exists a universal constant $c$ such that
%\begin{align*}
%|T(x,t)|^2_{g(x,t)}\leq \max\left\{cK,\max_{x\in M}|T(x,0)|^2_{g(x,0)}\right\}
%,\quad
%x\in M \ and\ t\in[0,\tau).
%\end{align*}
%\end{theorem}

\begin{theorem}\label{thm-HSdim2-T2}
For a compact complex surface $(M^4,J,g(t))$ that admits a pluriclosed flow on $[0,\tau)$ (where $\tau\leq+\infty$), if the Chern Ricci curvature satisfies
\begin{align*}
|(\Ric+S)(x,t)|_{g(x,t)}\leq K
,\quad 
x\in M \ and\ t\in[0,\tau)
\end{align*}
then there exists a universal constant $c$ such that
\begin{align*}
|T(x,t)|^2_{g(x,t)}\leq \max\left\{cK,\max_{x\in M}|T(x,0)|^2_{g(x,0)}\right\}
,\quad
x\in M \ and\ t\in[0,\tau).
\end{align*}
\end{theorem}

Notice that the second Chern Ricci $S_{i\bar{j}}=g^{p\bar{q}}\C_{p\bar{q}i\bar{j}}$ and the classical Ricci curvature
$\Ric_{i\bar{j}}=g^{p\bar{q}}\C_{i\bar{j}p\bar{q}}$ are different for non-K{\"a}hler metrics.
%Here the Chern scalar curvature is defined by $R=g^{\bar{j}i}g^{\bar{q}p}\C_{i\bar{j}p\bar{q}}$.

Next we consider a special case.
To begin, let us revisit the concept of Hermitian-symplectic form, which was introduced by Streets and Tian in \cite{MR2673720}.
A Hermitian-symplectic form, denoted as $H$, is a symplectic form with positive (1,1)-part.
It can be expressed by bi-degree as $H=\me+\p+\ov{\p}$, and it is easy to see that the (1,1)-part $\me$ is a pluriclosed metric.
It is worth noting that the pluriclosed flow preserves the Hermitian-symplectic condition (e.g. see \cite{MR3319954,2022arXiv220712643Y}).
In the Hermitian-symplectic case, we can derive a monotonic quantity and apply it to prove that all Hermitian-symplectic solitons are in fact K{\"a}hler Ricci solitons.
Further information about various types of solitons of pluriclosed flow can be found in \cite{MR4023384,MR4287690}.

Here is an outline of the rest paper.
In Section \ref{sec-pre}, we introduce notations that will be used later and review some facts of pluriclosed flow.
In Section \ref{sec-BBS}, we complete the proof of higher order derivative estimates (Theorem \ref{thm-BBS-curvature}) and the torsion estimate in dimension 2 (Theorem \ref{thm-HSdim2-T2}).
In Section \ref{sec-HSflow}, we consider the Hermitian-symplectic case.
We will introduce a monotonic quantity and apply it to Hermitian-symplectic solitons.
\\

{\bf Acknowledgments.} I want to express my gratitude to my advisor, Professor Gang
Tian, for his helpful suggestions and patient guidance.
I thanks Professor Jeffery D. Streets for his helpful comments.
And I thanks Xilun Li for interesting discussion.

\section{Preliminary}\label{sec-pre}
In this section, we recall some basic facts about pluriclosed flow and make some agreements on the notations.

\subsection{Pluriclosed flow}
First of all, we recall a equivalent definition of pluriclosed flow that will be used later.

We use $\C$ and $T$ to denote the curvature tensor and torsion tensor of the Chern connection $\n$, respectively.
And for convenience, we use both $g$ and the K{\"a}hler form $\me$ to represent the metric.
\begin{definition}
The pluriclosed flow can also be written as
\begin{align*}
\dt g=-S+Q
,\quad
g(0)=g_0
\end{align*}
where $S_{i\bar{j}}=g^{\bar{b}a}\C_{a\bar{b}i\bar{j}}$ and 
$Q_{i\bar{j}}=g^{\bar{b}a}g^{\bar{t}s}T_{ia\bar{t}}T_{\bar{j}\bar{b}s}$.
And the initial data $g_0$ is a pluriclosed metric.
\end{definition}

\begin{remark}
The tensor $Q$ is non-negative, which means that, in local coordinates, it is a non-negative Hermitian matrix.
And we have $\tr Q=g^{\bar{b}a}Q_{a\bar{b}}=|T|^2$.
\end{remark}

\begin{remark}
In local coordinates, the torsion tensor is $T_{ij\bar{k}}=\dd_{i}g_{j\bar{k}}-\dd_{j}g_{i\bar{k}}$.
This is because
\begin{align*}
T_{ij\bar{k}}=T_{ij}^{a}g_{a\bar{k}}
=(\Ga_{ij}^{a}-\Ga_{ji}^{a})g_{a\bar{k}}
=(\dd_{i}g_{j\bar{b}}g^{\bar{b}a}-\dd_{j}g_{i\bar{b}}g^{\bar{b}a})
	g_{a\bar{k}}
=\dd_{i}g_{j\bar{k}}-\dd_{j}g_{i\bar{k}},
\end{align*}
where $\Ga_{ij}^{a}=\dd_{i}g_{j\bar{b}}g^{\bar{b}a}$ is the Christoffel symbol of the Chern connection.
\end{remark}

Next we	would like to reformulate the pluriclosed condition $\dd\db\me=0$.
\begin{lemma}\label{lem-PCcondition}
A metric is pluriclosed if and only if
\begin{align*}
\C_{i\bar{j}p\bar{q}}-\C_{p\bar{j}i\bar{q}}
	-\C_{i\bar{q}p\bar{j}}+\C_{p\bar{q}i\bar{j}}
=g^{\bar{b}a}T_{ip\bar{b}}T_{\bar{j}\bar{q}a}
\end{align*}
\end{lemma}
\begin{proof}
Recalling the definition of the Chern curvature, we have
\begin{align*}
\C_{i\bar{j}p\bar{q}}
&=-\dd_{\bar{j}}\Ga_{ip}^{a}g_{a\bar{q}}
=-\dd_{\bar{j}}(\dd_{i}g_{p\bar{b}}g^{\bar{b}a})g_{a\bar{q}}
=-\dd_{i\bar{j}}g_{p\bar{q}}
	+g^{\bar{b}a}\dd_{i}g_{p\bar{b}}\dd_{\bar{j}}g_{a\bar{q}}
\end{align*}
So we get
\begin{align*}
&\C_{i\bar{j}p\bar{q}}-\C_{p\bar{j}i\bar{q}}
	-\C_{i\bar{q}p\bar{j}}+\C_{p\bar{q}i\bar{j}}
\\
=&-(\dd_{i\bar{j}}g_{p\bar{q}}-\dd_{p\bar{j}}g_{i\bar{q}}
	-\dd_{i\bar{q}}g_{p\bar{j}}+\dd_{p\bar{q}}g_{i\bar{j}})
	\q
	+g^{\bar{b}a}
	(\dd_{i}g_{p\bar{b}}\dd_{\bar{j}}g_{a\bar{q}}
	-\dd_{p}g_{i\bar{b}}\dd_{\bar{j}}g_{a\bar{q}}
	-\dd_{i}g_{p\bar{b}}\dd_{\bar{q}}g_{a\bar{j}}
	+\dd_{p}g_{i\bar{b}}\dd_{\bar{q}}g_{a\bar{j}})
\end{align*}

The second row is
\begin{align*}
\text{\rm II}
&=g^{\bar{b}a}
	(T_{ip\bar{b}}\dd_{\bar{j}}g_{a\bar{q}}
	+T_{pi\bar{b}}\dd_{\bar{q}}g_{a\bar{j}})
=g^{\bar{b}a}T_{ip\bar{b}}
	(\dd_{\bar{j}}g_{a\bar{q}}
	-\dd_{\bar{q}}g_{a\bar{j}})
=g^{\bar{b}a}T_{ip\bar{b}}T_{\bar{j}\bar{q}a}
\end{align*}
Then we obtain
\begin{align*}
&\C_{i\bar{j}p\bar{q}}-\C_{p\bar{j}i\bar{q}}
	-\C_{i\bar{q}p\bar{j}}+\C_{p\bar{q}i\bar{j}}
	-g^{\bar{b}a}T_{ip\bar{b}}T_{\bar{j}\bar{q}a}
\\
=&-(\dd_{i\bar{j}}g_{p\bar{q}}-\dd_{p\bar{j}}g_{i\bar{q}}
	-\dd_{i\bar{q}}g_{p\bar{j}}+\dd_{p\bar{q}}g_{i\bar{j}})
\end{align*}
On the other hand, we have
\begin{align*}
\db\me=\sqrt{-1}\dd_{\bar{j}}g_{p\bar{q}} 
	d\bar{z}^{j}\wedge dz^{p} \wedge d\bar{z}^{q}
\end{align*}
and
\begin{align*}
\dd\db\me
&=\sqrt{-1}\dd_{i\bar{j}}g_{p\bar{q}}
	dz^{i}\wedge d\bar{z}^{j}\wedge dz^{p} \wedge d\bar{z}^{q}
\end{align*}
Thus $\dd\db\me=0$ is equivalent to
\begin{align*}
\dd_{i\bar{j}}g_{p\bar{q}}-\dd_{p\bar{j}}g_{i\bar{q}}
	-\dd_{i\bar{q}}g_{p\bar{j}}+\dd_{p\bar{q}}g_{i\bar{j}}
	=0
\end{align*}
Combining above together, we conclude that a metric is pluriclosed if and only if
\begin{align*}
\C_{i\bar{j}p\bar{q}}-\C_{p\bar{j}i\bar{q}}
	-\C_{i\bar{q}p\bar{j}}+\C_{p\bar{q}i\bar{j}}
	-g^{\bar{b}a}T_{ip\bar{b}}T_{\bar{j}\bar{q}a}
	=0
\end{align*}
\end{proof}

As an application, we obtain that
\begin{proposition}
Let $(M^{2n},J,g)$ be a complete complex manifold with a pluriclosed metric.
If it is Chern flat, then its universal covering is $\mathbb{C}^{n}$ and, up to scaling, $g$ pulls back to the standard metric on $\mathbb{C}^{n}$.
\end{proposition}
\begin{proof}
According to Lemma \ref{lem-PCcondition}, we know that $g$ is K{\"a}hler since it is Chern flat.
Then we complete our proof by the {\it Uniformization Theorem} for K{\"a}hler manifolds. 
\end{proof}

\subsection{Notations}\label{subs-notation}
In this subsection, we give some notations that will be used later.

In this paper, we use $\n$ to denote the Chern connection.
And for convenience, we denote the gradient of a tensor $A$ by 
$DA=\n A+\nb A$.
Here $\n A$ and $\nb A$ are the \lq\lq{}(1,0)-part\rq\rq{} and \lq\lq{}(0,1)-part\rq\rq{} of the gradient, respectively.

For easy of notations, we will use the same subscript to denote the trace by the metric.
For example, we may write $S_{i\bar{j}}=\C_{a\bar{a}i\bar{j}}$ and $Q_{i\bar{j}}=T_{ia\bar{b}}T_{\bar{j}\bar{a}b}$.
Under this convention, the Chern Laplacian is $\D A=\n_{a}\n_{\bar{a}}A$ and the \lq\lq{}conjugation\rq\rq{} Chern Laplacian is $\Db A=\n_{\bar{a}}\n_{a}A$.

Next, we give two definitions that are often used in later calculations.

\begin{definition}\label{def-h[]}
Let $h=h_{i\bar{j}}$ be a tensor satisfying 
$\ov{h_{i\bar{j}}}=h_{j\bar{i}}$.
Then for any tensor $A=A_{p_1\cdots p_s \bar{q}_1\cdots \bar{q}_t}$, we define a tensor $h[A]$ by
\begin{align*}
(h[A])_{p_{1} \cdots p_s \bar{q}_1 \cdots \bar{q}_t}
&=-\sum_{\alpha=1}^{s}h_{p_{\alpha}\bar{a}}
	A_{p_1\cdots a \cdots p_s \bar{q}_1\cdots \bar{q}_t}
	-\sum_{\beta=1}^{t}h_{a\bar{q}_\beta}
	A_{p_1 \cdots p_s \bar{q}_1 \cdots \bar{a} \cdots \bar{q}_t}
\end{align*}
\end{definition}

Using this notation, we can write the variation of the square norm of a tensor as
\begin{align*}
\dt|A|^2
&=(\dt A,A)+(A,\dt A)+(A,(\dt g)[A])
\end{align*}
where $(\cdot,\cdot)$ is the Hermitian inner product induced by $g(t)$.

\begin{remark}
Since the tensor $h$ in the Definition \ref{def-h[]} is Hermitian , we have 
\begin{align*}
(A,h[A])=(h[A],A).
\end{align*}
\end{remark}

\begin{definition}\label{def-circle}
Given a tensor $Z=Z_{\bullet i\bar{j}}$.
Then for any tensor $A=A_{p_1\cdots p_s \bar{q}_1\cdots \bar{q}_t}$, we define a tensor $Z\ci A$ by
\begin{align*}
(Z\ci A)_{\bullet p_1 \cdots p_s \bar{q}_1\cdots \bar{q}_t}
&=-\sum_{\alpha=1}^{s}Z_{\bullet p_{\alpha}\bar{a}}
	A_{p_1\cdots a \cdots p_s \bar{q}_1\cdots \bar{q}_t}
	+\sum_{\beta=1}^{t}Z_{\bullet a\bar{q}_\beta}
	A_{p_1 \cdots p_s \bar{q}_1 \cdots \bar{a} \cdots \bar{q}_t}
\end{align*}
\end{definition}

This operator can be regarded as a generalization of the commutator of derivative.
For example, if $A=A_{i\bar{j}}$ we have
\begin{align*}
(\n_{p}\n_{\bar{q}}-\n_{\bar{q}}\n_{p})A_{i\bar{j}}
=-\C_{p\bar{q}ia}A_{a\bar{j}}+\C_{p\bar{q}a\bar{j}}A_{i\bar{a}}
=(\C\circ A)_{p\bar{q}i\bar{j}}
\end{align*}
In fact, it is true for any tensor $A\in T^{p,q}M$.

Formally, we have
\begin{lemma}\label{lem-ex-LaplacianAndConj}
For any tensor $A\in T^{p,q}M$, we have
\begin{align*}
(\n_{p}\n_{\bar{q}}-\n_{\bar{q}}\n_{p})A
=\C_{p\bar{q}}\circ A.
\end{align*}
In particular, we have
\begin{align*}
(\D-\Db)A=S\circ A.
\end{align*}
\end{lemma}
\begin{proof}
It is easy to check the first equation by direct computation.
Taking trace with respect to $p$,$\bar{q}$ and recalling the definition of $S$, we obtain the second equation.
\end{proof}

\begin{remark}\label{re-circAnd[]}
It is easy to see that if $A\in T^{p,0}M$, we have $h[T]=h\circ A$.
And since $h$ is Hermitian, we assume $h_{i\bar{j}}(x)=\lambda_i(x)\delta_{ij}$ and $g_{i\bar{j}}(x)=\delta_{ij}$ in a normal coordinates at a fixed point $x$.
Then we have
\begin{align*}
(h\ci A)_{p_1\cdots p_s}
=-\sum_{\alpha=1}^{s}h_{p_\alpha\bar{a}}
	A_{p_1 \cdots a \cdots p_s}
=-\sum_{\alpha=1}^{s}\lambda_{p_\alpha}(x)
	A_{p_1 \cdots p_s}
\end{align*}
And if $h$ is non-negative, we have
\begin{align*}
-(A,h\ci A)=-(A,h[A])\geq 0.
\end{align*}
\end{remark}
More detail of the calculation about those two definitions can be found in \cite{2022arXiv221204060Y}.

\section{Derivative estimates}\label{sec-BBS}
\subsection{Evolution equations}
In this subsection, we collect some evolution equations of higher derivative of curvature and torsion.

Firstly, we state the first and second Bianchi identities of Chern connection.
\begin{lemma}[Bianchi identity]
For Chern connection, we have the first Bianchi identity
\begin{align*}
\n_{\bar{j}}T_{ip\bar{q}}
=\C_{p\bar{j}i\bar{q}}-\C_{i\bar{j}p\bar{q}}
,\quad
\n_{i}T_{\bar{j}\bar{q}p}
=\C_{i\bar{q}p\bar{j}}-\C_{i\bar{j}p\bar{q}}
\end{align*}
and the second Bianchi identity
\begin{align*}
\n_{s}\C_{i\bar{j}p\bar{q}}
=\n_{i}\C_{s\bar{j}p\bar{q}}+T_{is}^{a}\C_{a\bar{j}p\bar{q}}
,\quad
\n_{\bar{k}}\C_{i\bar{j}p\bar{q}}
=\n_{\bar{j}}\C_{i\bar{k}p\bar{q}}
	+T_{\bar{j}\bar{k}}^{\bar{b}}\C_{i\bar{b}p\bar{q}}
\end{align*}
\end{lemma}

And then we compute the variation of Christoffel symbol.
\begin{lemma}\label{lem-EvChristoffel}
\begin{align*}
\dt\Ga_{ij}^{k}
&=-\n_{i}S_{j\bar{a}}g^{\bar{a}k}
	+\n_{i}Q_{j\bar{a}}g^{\bar{a}k}
\end{align*}
Briefly,
\begin{align*}
\dt\Ga=\n\C*1+\C*T+\n T*\Tb
\end{align*}
\end{lemma}
\begin{proof}
\begin{align*}
\dt\Ga_{ij}^{k}
&=\dt(\dd_{i}g_{j\bar{a}}g^{\bar{a}k})
=\dd_{i}\dt g_{j\bar{a}}g^{\bar{a}k}
	-\dd_{i}g_{j\bar{a}}g^{\bar{a}s}\dt g_{s\bar{l}}g^{\bar{l}k}
=\dd_{i}\dt g_{j\bar{a}}g^{\bar{a}k}
	-\Ga_{ij}^{s}\dt g_{s\bar{a}}g^{\bar{a}k}
\\
&=\n_{i}(\dt g_{j\bar{a}})g^{\bar{a}k}
=-\n_{i}S_{j\bar{a}}g^{\bar{a}k}
	+\n_{i}Q_{j\bar{a}}g^{\bar{a}k}
\end{align*}
Using Bianchi identity, we get
\begin{align*}
\n_{i}Q_{j\bar{a}}
&=\n_{i}T_{jb\bar{c}}T_{\bar{a}\bar{b}c}
	+T_{jb\bar{c}}\n_{i}T_{\bar{a}\bar{b}c}
=\n_{i}T_{jb\bar{c}}T_{\bar{a}\bar{b}c}
	+\C_{i\bar{b}c\bar{a}}T_{jb\bar{c}}
	-\C_{i\bar{a}c\bar{b}}T_{jb\bar{c}}
\end{align*}
Thus we obtain
\begin{align*}
\dt\Ga=\n\C*1+\C*T+\n T*\Tb
\end{align*}
\end{proof}

Next, we give the evolution equations of curvature and torsion.
And we leave the tedious calculation in the appendix.

\begin{lemma}\label{lem-EvCurvature}
\begin{align*}
\heat\C_{i\bar{j}p\bar{q}}
&=\n_{a}\C_{i\bar{b}p\bar{q}}	T_{\bar{a}\bar{j}b}
	-\n_{i}\C_{a\bar{j}p\bar{b}}T_{\bar{q}\bar{a}b}
	+\n_{i}\C_{p\bar{j}a\bar{b}}T_{\bar{q}\bar{a}b}
	+\n_{\bar{j}}\C_{b\bar{a}p\bar{q}}T_{ai\bar{b}}	
	-\n_{\bar{j}}\C_{i\bar{a}b\bar{q}}T_{pa\bar{b}}
	\q
	+\n_{\bar{j}}\C_{i\bar{q}b\bar{a}}T_{pa\bar{b}}
	+\C_{a\bar{j}i\bar{b}}\C_{b\bar{a}p\bar{q}}
	-\C_{a\bar{j}b\bar{a}}\C_{i\bar{b}p\bar{q}}
	+\C_{a\bar{j}p\bar{b}}\C_{i\bar{a}b\bar{q}}
	-\C_{a\bar{j}b\bar{q}}\C_{i\bar{a}p\bar{b}}
	\q
	+\C_{i\bar{j}a\bar{b}}\C_{b\bar{a}p\bar{q}}
	-\C_{a\bar{j}i\bar{b}}\C_{b\bar{a}p\bar{q}}
	+\C_{a\bar{j}b\bar{a}}\C_{i\bar{b}p\bar{q}}
	-\C_{a\bar{a}b\bar{j}}\C_{i\bar{b}p\bar{q}}
	-\C_{a\bar{j}p\bar{b}}\C_{i\bar{a}b\bar{q}}
	\q
	+\C_{p\bar{j}a\bar{b}}\C_{i\bar{a}b\bar{q}}
	+\C_{a\bar{j}p\bar{b}}\C_{i\bar{q}b\bar{a}}
	-\C_{p\bar{j}a\bar{b}}\C_{i\bar{q}b\bar{a}}	
	-\C_{i\bar{j}p\bar{a}}S_{a\bar{q}}
	-\C_{i\bar{j}p\bar{c}}T_{ca\bar{b}}T_{\bar{q}\bar{a}b}
	\q
	-\C_{i\bar{j}a\bar{c}}T_{pc\bar{b}}T_{\bar{q}\bar{a}b}
	+\C_{i\bar{j}c\bar{b}}T_{pa\bar{c}}T_{\bar{q}\bar{a}b}
	+\C_{i\bar{j}p\bar{a}}Q_{a\bar{q}}
	-\n_{i}T_{pa\bar{b}}\n_{\bar{j}}T_{\bar{q}\bar{a}b}
\end{align*}
Briefly,
\begin{align*}
\heat\C
&=\n\C*\Tb+\nb\C*T+\C*\C+\C*T*\Tb+\n T*\ov{\n T}
\end{align*}
\end{lemma}

\begin{lemma}\label{lem-EvTorsion}
\begin{align*}
\heat T_{ij\bar{k}}
&=-\C_{b\bar{k}i\bar{a}}T_{ja\bar{b}}
	+\C_{b\bar{k}j\bar{a}}T_{ia\bar{b}}
	+T_{ib\bar{c}}T_{\bar{a}\bar{k}c}T_{ja\bar{b}}
	-T_{jb\bar{c}}T_{\bar{a}\bar{k}c}T_{ia\bar{b}}
	\q
	-S_{a\bar{k}}T_{ij\bar{a}}
	+\n_{i}T_{ja\bar{b}}T_{\bar{k}\bar{a}b}
	-\n_{j}T_{ia\bar{b}}T_{\bar{k}\bar{a}b}
	+T_{ij\bar{a}}Q_{a\bar{k}}
\end{align*}
Briefly,
\begin{align*}
\heat T
&=\C*T+\n T*\Tb+T*T*\Tb
\end{align*}
\end{lemma}

To compute higher order derivative, we need the following two lemmas.
\begin{lemma}\label{lem-exdtAndconnection}
\begin{align*}
&[\dt,\n]A=\n\C*A+\C*T*A+\n T*\Tb*A
\\
&[\dt,D]A=D\C*A+DT*T*A
\end{align*}
where $D=\n+\nb$
\end{lemma}
\begin{proof}
By direct computation,
\begin{align*}
\dt\n_{i}A_{p_1 \cdots p_l \bar{q}_1 \cdots \bar{q}_k}
&=\dt\left(\dd_{i}A_{p_1 \cdots p_l \bar{q}_1 \cdots \bar{q}_k}
	-\sum_{\alpha=1}^{s}\Ga_{ip_\alpha}^{a}
	A_{p_1 \cdots a \cdots p_l \bar{q}_1 \cdots \bar{q}_k}\right)
\\
&=(\n\dt A)_{p_1 \cdots p_l \bar{q}_1 \cdots \bar{q}_k}
	-\sum_{\alpha=1}^{s}\dt\Ga_{ip_\alpha}^{a}
	A_{p_1 \cdots a \cdots p_l \bar{q}_1 \cdots \bar{q}_k}
\end{align*}
Applying Lemma \ref{lem-EvChristoffel}, we obtain
\begin{align*}
(\dt\n-\n\dt)A=\n\C*A+\C*T*A+\n T*\Tb*A
\end{align*}
Similarly, we have
\begin{align*}
(\dt D-D\dt)A=D\C*A+DT*T*A
\end{align*}
\end{proof}

\begin{remark}
From the first Bianchi identity, we have
\begin{align*}
\n\Tb*1=\C*1.
\end{align*}
But because the norm of $T$ and $\Tb$ are the same, we will not distinguish $A*T$ and $A*\Tb$ in the following when the first Bianchi identity are not required.
\end{remark}

\begin{lemma}\label{lem-ex-LaplaceAndconnection}
\begin{align*}
&[\n,\D]A=\n\C*A+\C*\n A+\n\nb A*T
\\
&[\nb,\D]A=\C*\nb A+\n\nb A*\Tb
\\
&[D,\D]A=D\C*A+\C*DA+D^2A*T
\end{align*}
\end{lemma}
\begin{proof}
Firstly, we claim that
\begin{align}\label{eq-2ndDerivativeExchange}
\n^{2}_{xy}-\n^{2}_{yx}=\C_{xy}-\n_{T(x,y)}.
\end{align}
Recall that
\begin{align*}
\n^{2}_{xy}=\n_{x}\n_{y}-\n_{\n_{x}y}
\end{align*}
and then by direct computation, we have
\begin{align*}
\n^{2}_{xy}-\n^{2}_{yx}
&=\n_{x}\n_{y}-\n_{\n_{x}y}-\n_{y}\n_{x}+\n_{\n_{y}x}
\\
&=\C_{xy}+\n_{[x,y]}-\n_{\n_{x}y}+\n_{\n_{y}x}
\\
&=\C_{xy}-\n_{T(x,y)}.
\end{align*}
We will use the notation defined in Definition \ref{def-circle}.
\begin{align*}
(\n^{3}A)_{a\bar{a}s}
&=(\n_{a}\n^{2}A)_{\bar{a}s}
=\n_{a}(\n^{2}_{\bar{a}s}A)-\n^{2}_{\bar{a}(\n_{a}s)}A
\\
&=\n_{a}\left(\n^{2}_{s\bar{a}}A-(\C\ci A)_{s\bar{a}}\right)
	-\n^{2}_{\bar{a}(\n_{a}s)}A
\\
&=(\n^{3}A)_{as\bar{a}}
	+\n^{2}_{(\n_{a}s)\bar{a}}A
	-((\n_{a}\C)\ci A)_{s\bar{a}}
	-(\C\ci (\n_{a}A))_{s\bar{a}}
	\q
	-(\C\ci A)_{(\n_{a}s)\bar{a}}
	-\n^{2}_{\bar{a}(\n_{a}s)}A
\\
&=(\n^{3}A)_{as\bar{a}}
	-((\n_{a}\C)\ci A)_{s\bar{a}}
	-(\C\ci (\n_{a}A))_{s\bar{a}}
\end{align*}
And from equation \eqref{eq-2ndDerivativeExchange}, we get
\begin{align*}
\n^{2}_{as}\nb A
&=\n^{2}_{sa}\nb A-T_{as}^{b}\n_{b}\nb A
\end{align*}
So
\begin{align*}
(\n^{3}A)_{a\bar{a}s}
&=(\n^{3}A)_{sa\bar{a}}-T_{as}^{b}\n_{b}\n_{\bar{a}}A
	-((\n_{a}\C)\ci A)_{s\bar{a}}
	-(\C\ci (\n_{a}A))_{s\bar{a}}
\end{align*}
Taking trace we obtain
\begin{align*}
[\n,\D]A=\n\C*A+\C*\n A+\n\nb A*T
\end{align*}
Similarly,
\begin{align*}
(\n^{3}A)_{a\bar{a}\bar{k}}
&=(\n_{a}\n^2A)_{\bar{a}\bar{k}}
=\n_{a}(\n^2_{\bar{a}\bar{k}}A)
=\n_{a}(\n^2_{\bar{k}\bar{a}}A-T_{\bar{a}\bar{k}}^{\bar{b}}\n_{\bar{b}}A)
\\
&=(\n^{3}A)_{a\bar{k}\bar{a}}
	-\n_{a}T_{\bar{a}\bar{k}}^{\bar{b}}\n_{\bar{b}}A
	-T_{\bar{a}\bar{k}}^{\bar{b}}\n_{a}\n_{\bar{b}}A
\\
&=(\n^{3}A)_{a\bar{k}\bar{a}}
	-(\C_{a\bar{k}b\bar{a}}-\C_{a\bar{a}b\bar{k}})\n_{\bar{b}}A
	-T_{\bar{a}\bar{k}}^{\bar{b}}\n_{a}\n_{\bar{b}}A
\end{align*}
The last line we use the first Bianchi identity.
Notice that
\begin{align*}
\n^{2}_{a\bar{k}}\nb A
&=\n^{2}_{\bar{k}a}\nb A+\C_{a\bar{k}}\ci\nb A.
\end{align*}
Thus
\begin{align*}
[\nb,\D]A=\C*\nb A+\n\nb A*\Tb
\end{align*}
and
\begin{align*}
[D,\D]A=D\C*A+\C*DA+D^2A*T
\end{align*}
\end{proof}

Then we make a remark on the higher order derivative of torsion.
Using the first Bianchi identity, we find that 
\begin{align*}
DT*1=\n T*1+\nb T*1
=\n T*1+\C*1
\end{align*}		
Similarly, for the second order derivative, we have
\begin{align*}
D^2T*1=\n^2T*1+D\C*1+\C*T
\end{align*}
In fact, the higher order derivative of torsion $D^mT$ can be expressed as the sum of \lq\lq{}partial\rq\rq{} derivative $\n^mT$ and lower derivative of torsion and curvature.
Formally, we have

\begin{lemma}\label{lem-torsion-totalhigherDerivativeAndPartial}
We have
\begin{align*}
D^{m}T*1
&=\n^{m}T*1+D^{m-1}\C*1+D^{m-2}(\C*T)
\end{align*}
for $m\geq2$.
\end{lemma}
\begin{proof}
First of all, we claim that
\begin{align}\label{eq-nbnmTorsion}
\nb\n^{m}T*1=\n^{m}\C*1+\n^{m-1}(\C*T)
\end{align}
We prove this claim by induction on $m$.
Exchanging the order of derivative and applying the first Bianchi identity, we can prove the case of $m=1$.
Then we assume equation \eqref{eq-nbnmTorsion} is true for $m$.
By direct computation,
\begin{align*}
\nb\n^{m+1}T*1
&=\C*\n^{m}T+\n\nb\n^{m}T*1
\\
&=\C*\n^{m}T+\n(\n^{m}\C*1+\n^{m-1}(\C*T))
\\
&=\n^{m+1}\C*1+\n^{m}(\C*T)
\end{align*}
So we prove our claim.

Then we prove this lemma by induction on $m$.
The case of $m=2$ have been checked.
By the induction hypothesis,
\begin{align*}
D^{m+1}T*1
&=D(\n^{m}T*1+D^{m-1}\C*1+D^{m-2}(\C*T))
\\
&=\n^{m+1}T*1+\nb\n^{m}T*1
	+D^{m}\C*1+D^{m-1}(\C*T)
\\
&=\n^{m+1}T*1+D^{m}\C*1+D^{m-1}(\C*T)
\end{align*}
The last line is because of equation \eqref{eq-nbnmTorsion}.
\end{proof}

Now we give evolution equations of higher order derivative of curvature and torsion.
And we leave the calculation in the appendix.

\begin{lemma}\label{lem-Ev-higherDerivativeCurvature}
We have
\begin{align*}
\heat D^{m}\C
&=\sum_{k=1}^{m+1}D^{k}\C*D^{m+1-k}T
	+\sum_{k=1}^{m+1}\n^{k}T*\nb^{m+2-k}\Tb
	+D^{m}\left(\C^{*2}+\C*T^{*2}\right)
\end{align*}
for $m\geq0$.
\end{lemma}

\begin{lemma}\label{lem-ev-highOrderDerivativeTorsion}
We have
\begin{align*}
\heat\n^{m}T
&=\n^{m}(\C*T+T^{*2}*\Tb)+\n^{m+1}(T*\Tb)
\end{align*}
for $m\geq0$.
\end{lemma}

\subsection{Higher derivative estimates}
In this subsection, we will give the higher derivative estimates.
Firstly, we state a lemma.

\begin{lemma}\label{lem-ev-norm}
We have
\begin{align*}
\heat|A|^2
&=(\heat A,A)+(A,\heat A)-|DA|^2
	\q
	+(A,Q[A])+(A,S\circ A)-(A,S[A])
\end{align*}
Briefly,
\begin{align*}
\heat|A|^2
&=\heat A*A+\C*A^{*2}+T^{*2}*A^{*2}-|DA|^2
\end{align*}
\end{lemma}
\begin{proof}
Using notations defined in subsection \ref{subs-notation}, we have
\begin{align*}
\dt|A|^2
&=(\dt A,A)+(A,\dt A)+(A,(-S+Q)[A])
%\\
%&=(\dt A,A)+(A,\dt A)+\C*A^{*2}+T^{*2}*A^{2}
\end{align*}
On the other hand,
\begin{align*}
\D|A|^2
&=(\D A,A)+(A,\Db A)+|\n A|^2+|\nb A|^2
\\
&=(\D A,A)+(A,\D A)-(A,S\ci A)+|DA|^2
%\\
%&=(\D A,A)+(A,\D A)+|DA|^2+\C*A^{*2}
\end{align*}
The second line is because of Lemma \ref{lem-ex-LaplacianAndConj}.
Combining above together we obtain
\begin{align*}
\heat|A|^2
&=(\heat A,A)+(A,\heat A)-|DA|^2
	\q
	+(A,Q[A])+(A,S\circ A)-(A,S[A])
%&=(\heat A,A)+(A,\heat A)-|DA|^2+\C*A^{*2}+T^{*2}*A^{*2}
%\\
%&=\heat A*A+\C*A^{*2}+T^{*2}*A^{*2}-|DA|^2.
\end{align*}
\end{proof}

To give the derivative estimates, we need to slightly modify the standard argument of the Ricci flow (e.g. see \cite{MR2061425,MR2274812}). 

\begin{proof}[Proof of Theorem \ref{thm-BBS-curvature}.]
This proof is by induction on $m$.
For reading convenience, we divide it into four steps.

{\bf Step 1.} Compute evolution equations of $|\C|^2$ and $|T|^2$.

From the pluriclosed condition (Lemma \ref{lem-PCcondition}), we know that $|T|^2\leq c|\C|\leq cK$, where $c$ is a constant depending only on the dimension.
For convenience, we will use $C$ to denote a constant depending only on the dimension $n$ and $\max\{\tau,1\}$ throughout this proof.

From Lemma \ref{lem-EvCurvature} and Lemma \ref{lem-ev-norm}, we know that
\begin{align}
\heat|\C|^2
&\leq C\Big(
	|\C||D\C||T|+|\C||\n T|^2+|\C|^3+|\C|^2|T|^2
	\Big)
	-|D\C|^2\label{eq-ev-curvature2}
\\
&\leq C(K^3+K|\n T|^2)-\frac{1}{2}|D\C|^2\nonumber
\end{align}
We use the Cauchy-Schwarz inequality in the last line.
And similarly, from Lemma \ref{lem-EvTorsion}, we have
\begin{align}
\heat|T|^2
&\leq C\Big(
	|\C||T|^2+|\n T||T|^2+|T|^4
	\Big)
	-|\n T|^2-|\nb T|^2\label{eq-ev-torsion2}
\\
&\leq CK^2-\frac{1}{2}|\n T|^2\nonumber
\end{align}
Let $\alpha$ be a constant determined later.
We define a function by $F_0=|\C|^2+\alpha|T|^2$ and we have
\begin{align*}
\heat F_0
&\leq CK^2(\beta+K)+(CK-\frac{1}{2}\alpha)|\n T|^2
	-\frac{1}{2}|D\C|^2 
\end{align*}
So we can choose $\alpha=2CK+1$ such that
\begin{align}\label{eq-ev-F0}
\heat F_0
&\leq CK^3-\frac{1}{2}(|D\C|^2+|\n T|^2)
\end{align}

{\bf Step 2.} Estimate $|DT|$.

We begin with the evolution equation of $|\n T|^2$.
The first Bianchi identity $\n\Tb*1=\C*1$ will be used in the calculation.
From Lemma \ref{lem-ev-highOrderDerivativeTorsion}, we have
\begin{align*}
\heat|\n T|^2
&\leq C\Big(
	|\n T||\n\C||T|+|\n T|^2|\C|
	+|\n T|^2|T|^2+|\n T||\n\Tb||T|^2
	\q
	+|\n^2T||\n T||T|+|\n T|^2|\n\Tb|+|\n T||\n^2\Tb||T|
	\Big)
	-|\n^2 T|^2-|\nb\n T|^2
\\
&\leq C\Big(
	|\n T||\n\C||T|+|\n T|^2|\C|
	+|\n T|^2|T|^2+|\n T||\C||T|^2
	\q
	+|\n^2T||\n T||T|
	\Big)
	-|\n^2 T|^2-|\nb\n T|^2
\\
&\leq CK^3+CK|\n T|^2
	+\frac{1}{4}|\n\C|^2
	-\frac{3}{4}|\n^2 T|^2
\end{align*}
Let $\beta_1$ be a constant determined later.
We define a function $Z=t|\n T|^2+\beta_1 F_0$ and we have
\begin{align*}
\heat Z
&=|\n T|^2+t\heat|\n T|^2+\beta_1\heat F_0
\\
&\leq CK^3(\beta_1+t)
	+(1+CKt-\frac{1}{2}\beta_1)|\n T|^2
	+(\frac{1}{4}t-\frac{1}{2}\beta_1)|D\C|^2
	-\frac{3}{4}t|\n^2 T|^2
\end{align*}
We use equation \eqref{eq-ev-F0} in the last line.
Since $t\leq \tau/K^2$, we can choose
\begin{align*}
\beta=\max\left\{\frac{2C\tau}{K}+3,\frac{\tau}{K^2}+1
	\right\}
\end{align*}
such that
\begin{align}\label{eq-ev-Z}
\heat Z
&\leq CK^3-\frac{1}{2}(|\n T|^2+t|D\C|^2+t|\n^2T|).
\end{align}
Applying maximum principle to $Z$ and noticing $Z(x,0)=\beta(|\C|^2+\alpha |T|^2)$, we have
\begin{align*}
\max_{x\in M}Z(x,t)
\leq CK^3t+\max_{x\in M}Z(x,0)
\leq CK^2
\end{align*}
So we get $|\n T|\leq \frac{CK}{t^{1/2}}$.
Since $DT*1=\n T*1+\C*1$, we obtain the estimate
\begin{align*}
|DT|\leq \frac{C(n,\tau)K}{t^{1/2}}.
\end{align*}

{\bf Step 3.} Estimate $|D\C|$ and $|D^2T|$.

Firstly, we compute the evolution equation of $|D\C|^2$.
From Lemma \ref{lem-Ev-higherDerivativeCurvature} and Lemma \ref{lem-ev-norm}, we have
\begin{align*}
\heat |D\C|^2
&\leq C\Big(
	|D\C||D^2\C||T|+|D\C|^2|DT|
	+|D\C||\n T||\n^2T|
	\q
	+|D\C|^2|\C|+|D\C|^2|T|^2+|D\C||DT||\C||T|
	\Big)
	-|D^2\C|^2
\\
&\leq C\Big(
	|D\C|^2|T|^2+|D\C|^2(K^{-2}|DT|^2+K^2)
	+K^{-2}|D\C|^2|\n T|^2
	\q
	+|DT|^2|\C|^2
	\Big)
	+K^2|\n^2 T|^2
	-\frac{3}{4}|D^2\C|^2
\\
&\leq C(K^2+K^{-2}|DT|^2)|D\C|^2
	+CK^2|DT|^2
	+K^2|\n^2 T|^2-\frac{3}{4}|D^2\C|^2
\end{align*}
By the estimate of $|DT|$ obtained in Step 2, we have
\begin{align*}
\heat|D\C|^2
&\leq C(K^2+t^{-1})|D\C|^2
	+CK^4t^{-1}
	+K^2|\n^2 T|^2-\frac{3}{4}|D^2\C|^2
\end{align*}
Then we compute the evolution equation of $|\n^2T|^2$.
By Lemma \ref{lem-ev-highOrderDerivativeTorsion} and the first Bianchi identity,
\begin{align*}
\heat|\n^2T|^2
&\leq C\Big(
	|\n^2T||\n^2\C||T|+|\n^2T||\n\C||\n T|+|\n^2T|^2|\C|
	\q
	+|\n^2T|^2|T|^2+|\n^2T||\n^2\Tb||T|^2+|\n^2T||\n T|^2|T|
	+|\n^2T||\n T||\n\Tb||T|
	\q
	+|\n^2T||\n^3T||T|+|\n^2T|^2|\n\Tb|+|\n^2T||\n T||\n^2\Tb|
	\q
	+|\n^2T||\n^3\Tb||T|
	\Big)
	-|\n^3T|^2-|\nb\n^2T|^2
\\
&\leq C\Big(
	|\n^2T||\n^2\C||T|+|\n^2T||\n\C||\n T|+|\n^2T|^2|\C|
	+|\n^2T|^2|T|^2
	\q
	+|\n^2T||\n\C||T|^2
	+|\n^2T||\n T|^2|T|+|\n^2T||\n T||\C||T|
	\q
	+|\n^2T||\n^3T||T|
	\Big)
	-|\n^3T|^2-|\nb\n^2T|^2
\\
&\leq C(K^2+K^{-2}|\n T|^2)|\n^2T|^2
	+C(|\n T|^4+K^2|\n T|^2)
	\q
	+K^2|\n\C|^2
	+\frac{1}{4}|\n^2\C|^2-\frac{3}{4}|\n^3T|^2
\end{align*}
Since $|\n T|\leq CKt^{-1/2}$, we obtain
\begin{align*}
\heat|\n^2T|^2
&\leq C(K^2+t^{-1})|\n^2T|^2
	+CK^4(t^{-2}+t^{-1})
	+K^2|\n\C|^2
	\q
	+\frac{1}{4}|\n^2\C|^2-\frac{3}{4}|\n^3T|^2
\end{align*}

Let $\beta_2$ be a constant determined later.
Define a function
\begin{align*}
F_1=t^2(|D\C|^2+|\n^2T|^2)+\beta_2Z
\end{align*}
and we have
\begin{align*}
\heat F_1
&=2t(|D\C|^2+|\n^2T|^2)+t^2\heat(|D\C|^2+|\n^2T|^2)+\beta_2\heat Z
\\
&\leq CK^4(\beta_2+1+t)
	+t\left(
		C(K^2t+1)-\frac{1}{2}\beta_2
		\right)(|D\C|^2+|\n^2T|^2)
	\q
	-\frac{1}{2}\beta_2|\n T|^2
	-\frac{1}{2}t^2(|D^2\C|^2+|\n^3T|^2)
\end{align*}
Here we use equation \eqref{eq-ev-Z}.
Since $t\leq \tau/K^2$, we can choose $\beta_2=2C(\tau+1)+1$ such that
\begin{align*}
\heat F_1
\leq CK^4-\frac{1}{2}\left(
	|\n T|^2+t(|D\C|^2+|\n^2T|^2)
	+t^2(|D^2\C|^2+|\n^3T|^2)
	\right)
\end{align*}
Similarly, applying the maximum principle and noticing $t\leq\tau/K^2$, we can obtain that
\begin{align*}
\max_{x\in M}F_1(x,t)\leq CK^4t+\max_{x\in M}F_1(x,0)
\leq CK^2
\end{align*}
Then we have
\begin{align*}
|D\C|,|\n^2T|\leq \frac{C(n,\tau)K}{t^{1/2}}
\end{align*}
By Lemma \ref{lem-torsion-totalhigherDerivativeAndPartial}, we obtain
\begin{align*}
|D^2T|\leq |\n^2T|+|D\C|+|\C*T|
\leq\frac{C(n,\tau)K}{t^{1/2}}.
\end{align*}

{\bf Step 4.} Estimate $|D^m\C|$ and $|D^{m+1}T|$.

We complete our proof by induction on $m$.
It is sufficient to prove the next claim.
\\
{\bf Claim.}
We can inductively define a function by 
\begin{align*}
F_m
&=t^{m+1}(|D^{m}\C|^2+|\n^{m+1}T|^2)
	+\ga_m F_{m-1}
,\quad m\geq 1
\end{align*}
such that
\begin{align}\label{eq-claim-condition1}
\heat F_m
&\leq CK^4-\frac{1}{2}\left(
	|\n T|^2
	+\sum_{k=1}^{m+1}t^{k}
	(|D^{k}\C|^2+|\n^{k+1}T|^2)
	\right)
\end{align}
and
\begin{align}\label{eq-claim-condition2}
\max_{x\in M}F_{m}(x,0)\leq CK^2.
\end{align}
Here $\ga_m$ is a constant determined later.
Recall that $F_0$ and $F_1$ have been defined earlier.

First of all, we compute the evolution equation of $|D^{m}\C|^2+|\n^{m+1}T|^2$.
By direct computation, we have
\begin{align*}
\heat|D^{m}\C|^2
&\leq C\Big(
	\sum_{k=1}^{m+1}|D^{m}\C||D^{k}\C||D^{m+1-k}T|
	+\sum_{k=1}^{m+1}|D^{m}\C||\n^kT||\n^{m+2-k}T|
	\q
	+\sum_{k=0}^{m}|D^{m}\C||D^k\C||D^{m-k}\C|
	+\sum_{0\leq k,l\leq m}|D^{m}\C||D^k\C||D^{l}T||D^{m-k-l}T|
	\q
	+|D^{m}\C|^2(|\C|+|T|^2)
	\Big)
	-|D^{m+1}\C|^2
\end{align*}
From the induction hypothesis
\begin{align*}
|D^k\C|\leq CKt^{-(k+1)/2}
,\quad 1\leq k \leq m-1
,\quad
|D^{k}T|\leq CKt^{-k/2}
,\quad 1\leq k \leq m
\end{align*}
we know that
\begin{align*}
\sum_{k=1}^{m-1}|D^{m}\C||D^{k}\C||D^{m+1-k}T|
\leq C|D^{m}\C|K^2t^{-\frac{m+2}{2}}
\leq C(t^{-1}|D^{m}\C|^2+K^4t^{-(m+1)});
\\
\sum_{k=2}^{m}|D^{m}\C||\n^kT||\n^{m+2-k}T|
\leq C|D^{m}\C|K^2t^{-\frac{m+2}{2}}
\leq C(t^{-1}|D^{m}\C|^2+K^4t^{-(m+1)});
\\
\sum_{k=1}^{m-1}|D^{m}\C||D^k\C||D^{m-k}\C|
\leq C|D^{m}\C|K^2t^{-\frac{m+2}{2}}
\leq C(t^{-1}|D^{m}\C|^2+K^4t^{-(m+1)});
\\
\sum_{0\leq k+l \leq m-1}|D^{m}\C||D^k\C||D^{l}T||D^{m-k-l}T|
\leq C|D^{m}\C|K^3t^{-\frac{m+1}{2}}
\leq C(K^2|D^{m}\C|^2+K^4t^{-(m+1)});
\end{align*}
Since $|DT|\leq CKt^{-1/2}$, we get
\begin{align*}
\heat|D^{m}\C|^2
&\leq C\Big(
	K|D^{m}\C|^2
	+|D^{m}\C||D^{m+1}\C||T|
	+|D^{m}\C|^2|DT|
	\q
	+|D^{m}\C||\n^{m+1}T||\n T|
	+t^{-1}|D^{m}\C|^2
	+K^4t^{-m-1}
	\Big)
	-|D^{m+1}\C|^2
\\
&\leq C(K^2+t^{-1})|D^m\C|^2+CK^{4}t^{-(m+1)}
	+K^2|\n^{m+1}T|^2-\frac{3}{4}|D^{m+1}\C|^2
\end{align*}
Here we use inequalities
\begin{align*}
|DT|=|DT|K^{-1}K\leq |DT|^2K^{-2}+K^2\leq t^{-1}+K^2
\end{align*}
and
\begin{align*}
C|D^{m}\C||\n^{m+1}T||\n T|
&=C(K^{-1}|D^{m}\C||\n T|)(K|\n^{m+1}T|)
\\
&\leq Ct^{-1}|D^{m}\C|^2+K^2|\n^{m+1}T|^2
\end{align*}
Similarly, for $|\n^{m+1}T|^2$ we have
\begin{align*}
\heat|\n^{m+1}T|^2
&\leq C\Big(
	\sum_{k=0}^{m+1}|\n^{m+1}T||\n^{k}\C||\n^{m+1-k}T|
	+\sum_{k=0}^{m+2}|\n^{m+1}T||\n^kT||\n^{m+2-k}\Tb|
	\q
	+\sum_{0\leq k,l\leq m+1}
		|\n^{m+1}T||\n^kT||\n^lT||\n^{m+1-k-l}\Tb|
	\q
	+|\n^{m+1}T|^2(|\C|+|T|^2)
	\Big)
	-|\n^{m+2}T|^2-|\nb\n^{m+1}T|^2
\end{align*}

Using the induction hypothesis and noticing the first Bianchi identity $\n\Tb*1=\C*1$, we have
\begin{align*}
&\sum_{k=1}^{m-1}|\n^{m+1}T||\n^{k}\C||\n^{m+1-k}T|
\leq C|\n^{m+1}T|K^2t^{-\frac{m+2}{2}}
\leq C(t^{-1}|\n^{m+1}T|^2+K^4t^{-(m+1)});
\\
&\sum_{k=2}^{m}|\n^{m+1}T||\n^kT||\n^{m+2-k}\Tb|
\leq C|\n^{m+1}T|K^2t^{-\frac{m+2}{2}}
\leq C(t^{-1}|\n^{m+1}T|^2+K^4t^{-(m+1)});
\\
&\sum_{1\leq k+l \leq m}
		|\n^{m+1}T||\n^kT||\n^lT||\n^{m+1-k-l}\Tb|
\leq C(K^2|\n^{m+1}T|^2+K^4t^{-(m+1)}).
\end{align*}
Thus
\begin{align*}
\heat|\n^{m+1}T|^2
&\leq C\Big(
	K|\n^{m+1}T|^2
	+|\n^{m+1}T||\n^{m+1}\C||T|+|\n^{m+1}T||\n^m\C||\n T|
	\q
	+|\n^{m+1}T||\n^{m+2}T||T|+t^{-1}|\n^{m+1}T|^2+K^4t^{-(m+1)}
	\Big)
	\q
	-|\n^{m+2}T|^2-|\nb\n^{m+1}T|^2
\\
&\leq CK^4t^{-(m+1)}
	+C(K^2+t^{-1})|\n^{m+1}T|^2
	+K^2|\n^{m}\C|^2
	\q
	+\frac{1}{4}|\n^{m+1}\C|^2
	-\frac{3}{4}|\n^{m+2}T|^2
\end{align*}
The second inequality is because of
\begin{align*}
C|\n^{m+1}T||\n^m\C||\n T|
&=C(K^{-1}|\n^{m+1}T||\n T|)(K|\n^m\C|)
\\
&\leq Ct^{-1}|\n^{m+1}T|^2+K^2|\n^m\C|^2
\end{align*}
Combining above together and using the induction hypothesis \eqref{eq-claim-condition1} for $F_{m-1}$, we obtain
\begin{align*}
\heat F_m
&=mt^m(|D^m\C|^2+|\n^{m+1}T|^2)
	+t^{m+1}\heat(|D^m\C|^2+|\n^{m+1}T|^2)
	\q
	+\ga_m\heat F_{m-1}
\\
&\leq CK^4(1+\ga_m)
	+t^m\left(
			C(1+K^2t)-\frac{1}{2}\ga_m
		\right)(|D^m\C|^2+|\n^{m+1}T|^2)
	\q
	-\frac{1}{2}\ga_m\left(|\n T|^2+\sum_{k=1}^{m-1}t^{k}
	(|D^{k}\C|^2+|\n^{k+1}T|^2)
	\right)
	\q
	-\frac{1}{2}t^{m+1}(|D^{m+1}\C|^2+|\n^{m+2}T|^2)
\end{align*}
Since $t\leq\tau/K^2$, we can choose $\beta_m=2C(\tau+1)+1$ such that
\begin{align*}
\heat F_m
&\leq CK^4-\frac{1}{2}\left(
	|\n T|^2+\sum_{k=1}^{m+1}t^{k}
	(|D^{k}\C|^2+|\n^{k+1}T|^2)
	\right)
\end{align*}
And by the induction hypothesis \eqref{eq-claim-condition2} for $F_{m-1}$, we know that
\begin{align*}
\max_{x\in M}F_{m}(x,0)\leq CK^2.
\end{align*}
So we prove the claim.
And by maximum principle, we get
\begin{align*}
\max_{x\in M}F_m(x,t)\leq CK^4t+\max_{x\in M}F_m(x,0)
\leq CK^4\frac{\tau}{K^2}+CK^2
\leq CK^2.
\end{align*}
Then we obtain
\begin{align*}
|D^m\C|,|\n^{m+1}T|\leq \frac{CK}{t^{(m+1)/2}}.
\end{align*}
From Lemma \ref{lem-torsion-totalhigherDerivativeAndPartial} and estimates we obtained, we have
\begin{align*}
|D^{m+1}T|
&\leq C(|\n^{m+1}T|+|D^m\C|
	+\sum_{k=0}^{m-1}|D^{k}\C||D^{m-1-k}T|)
\\
&\leq \frac{CK}{t^{(m+1)/2}}+\frac{CK^2}{t^{m/2}}
	=\frac{CK}{t^{(m+1)/2}}+\frac{CK^2t^{1/2}}{t^{(m+1)/2}}
\\
&\leq\frac{CK}{t^{(m+1)/2}}
\end{align*}
The last inequality is because of $t\leq\tau/K^2$.
And we complete this proof by induction.
\end{proof}

As an application, we prove the next result about maximal existence time of pluriclosed flow.

\begin{proof}[Proof of Theorem \ref{thm-maximalExistenceTime}.]
The argument is standard.
By Theorem \ref{thm-BBS-curvature}, we know that if the curvature is bounded on $[0,t_0]$, then the pluriclosed flow can smoothly pass the time point $t_0$.
So we obtain the result stated above since $\tau$ is the maximal existence time.
\end{proof}

\subsection{The surface case}\label{sec-dim2}

We are particularly interested in the pluriclosed flow on complex surfaces due to the geometrization program of complex surfaces introduced by Tian and Streets (see e.g. \cite{MR4181011,MR2755684}).
In this section, we give a torsion estimate using the Chern Ricci curvature in dimension 2.

Firstly, we calculate the evolution equation of $|T|^2$ more precisely.
\begin{lemma}
\begin{align*}
\heat|T|^2
&=4\C_{b\bar{k}j\bar{a}}T_{ia\bar{b}}T_{\bar{i}\bar{j}k}
	+4T_{ib\bar{c}}T_{ja\bar{b}}T_{\bar{a}\bar{k}c}T_{\bar{i}\bar{j}k}
	-2Q_{a\bar{i}}T_{ij\bar{k}}T_{\bar{a}\bar{j}k}
	+Q_{a\bar{k}}T_{ij\bar{a}}T_{\bar{i}\bar{j}k}
	\q
	+2\n_{i}T_{ja\bar{b}}T_{\bar{k}\bar{a}b}T_{\bar{i}\bar{j}k}
	+2\n_{\bar{i}}T_{\bar{j}\bar{a}b}T_{ka\bar{b}}T_{ij\bar{k}}
	-|\n T|^2-|\nb T|^2
\end{align*}
\end{lemma}
\begin{proof}
From Lemma \ref{lem-ev-norm}, we have
\begin{align*}
\heat|T|^2
&=(\heat T,T)+(T,\heat T)-|DT|^2
	\q
	+(T,Q[T])+(T,S\circ T)-(T,S[T])
\end{align*}

\text{\bf Step 1.} Compute terms in the second line.
By definition,
\begin{align*}
(S[T])_{ij\bar{k}}
&=-S_{i\bar{a}}T_{aj\bar{k}}
	-S_{j\bar{a}}T_{ia\bar{k}}
	-S_{a\bar{k}}T_{ij\bar{a}}
\\
(S\circ T)_{ij\bar{k}}
&=-S_{i\bar{a}}T_{aj\bar{k}}
	-S_{j\bar{a}}T_{ia\bar{k}}
	+S_{a\bar{k}}T_{ij\bar{a}}
\end{align*}
so we have
\begin{align*}
(-S[T]+S\circ T)_{ij\bar{k}}
=2S_{a\bar{k}}T_{ij\bar{a}}
\end{align*}
and
\begin{align*}
(T,-S[T]+S\circ T)
&=2T_{ij\bar{k}}\overline{S_{a\bar{k}}T_{ij\bar{a}}}
=2S_{k\bar{a}}T_{ij\bar{k}}T_{\bar{i}\bar{j}a}.
\end{align*}
Next we compute $(T,Q[T])$.
By definition,
\begin{align*}
(Q[T])_{ij\bar{k}}
&=-Q_{i\bar{a}}T_{aj\bar{k}}
	-Q_{j\bar{a}}T_{ia\bar{k}}
	-Q_{a\bar{k}}T_{ij\bar{a}},
\end{align*}
then we have
\begin{align*}
(T,Q[T])
&=-T_{ij\bar{k}}\left(\ov{Q_{i\bar{a}}T_{aj\bar{k}}
	+Q_{j\bar{a}}T_{ia\bar{k}}
	+Q_{a\bar{k}}T_{ij\bar{a}}}\right)
\\
&=-2Q_{a\bar{i}}T_{ij\bar{k}}T_{\bar{a}\bar{j}k}
	-Q_{k\bar{a}}T_{ij\bar{k}}T_{\bar{i}\bar{j}a}
\end{align*}
Notice that $(T,Q[T])\leq 0$ since $Q$ is non-negative.

\text{\bf Step 2.}
From Lemma \ref{lem-EvTorsion}, we get
\begin{align*}
(\heat T,T)
&=\Big\{-\C_{b\bar{k}i\bar{a}}T_{ja\bar{b}}
	+\C_{b\bar{k}j\bar{a}}T_{ia\bar{b}}
	+T_{ib\bar{c}}T_{\bar{a}\bar{k}c}T_{ja\bar{b}}
	-T_{jb\bar{c}}T_{\bar{a}\bar{k}c}T_{ia\bar{b}}
	\q
	-S_{a\bar{k}}T_{ij\bar{a}}
	+\n_{i}T_{ja\bar{b}}T_{\bar{k}\bar{a}b}
	-\n_{j}T_{ia\bar{b}}T_{\bar{k}\bar{a}b}
	+T_{ij\bar{a}}Q_{a\bar{k}}
	\Big\}T_{\bar{i}\bar{j}k}
\\
&=2\C_{b\bar{k}j\bar{a}}T_{ia\bar{b}}T_{\bar{i}\bar{j}k}
	+2T_{ib\bar{c}}T_{ja\bar{b}}T_{\bar{a}\bar{k}c}T_{\bar{i}\bar{j}k}
	-S_{a\bar{k}}T_{ij\bar{a}}T_{\bar{i}\bar{j}k}
	+Q_{a\bar{k}}T_{ij\bar{a}}T_{\bar{i}\bar{j}k}
	\q
	+2\n_{i}T_{ja\bar{b}}T_{\bar{k}\bar{a}b}T_{\bar{i}\bar{j}k}
\end{align*}
Similarly, we have
\begin{align*}
(T,\heat T)
&=2\C_{b\bar{k}j\bar{a}}T_{ia\bar{b}}T_{\bar{i}\bar{j}k}
	+2T_{ib\bar{c}}T_{ja\bar{b}}T_{\bar{a}\bar{k}c}T_{\bar{i}\bar{j}k}
	-S_{a\bar{k}}T_{ij\bar{a}}T_{\bar{i}\bar{j}k}
	+Q_{a\bar{k}}T_{ij\bar{a}}T_{\bar{i}\bar{j}k}
	\q
	+2\n_{\bar{i}}T_{\bar{j}\bar{a}b}T_{ka\bar{b}}T_{ij\bar{k}}
\end{align*}

\text{\bf Step 3.}
Combining above together, we obtain
\begin{align*}
\heat|T|^2
&=4\C_{b\bar{k}j\bar{a}}T_{ia\bar{b}}T_{\bar{i}\bar{j}k}
	+4T_{ib\bar{c}}T_{ja\bar{b}}T_{\bar{a}\bar{k}c}T_{\bar{i}\bar{j}k}
	-2S_{a\bar{k}}T_{ij\bar{a}}T_{\bar{i}\bar{j}k}
	+2Q_{a\bar{k}}T_{ij\bar{a}}T_{\bar{i}\bar{j}k}
	\q
	+2\n_{i}T_{ja\bar{b}}T_{\bar{k}\bar{a}b}T_{\bar{i}\bar{j}k}
	+2\n_{\bar{i}}T_{\bar{j}\bar{a}b}T_{ka\bar{b}}T_{ij\bar{k}}
	+2S_{k\bar{a}}T_{ij\bar{k}}T_{\bar{i}\bar{j}a}
	\q
	-2Q_{a\bar{i}}T_{ij\bar{k}}T_{\bar{a}\bar{j}k}
	-Q_{k\bar{a}}T_{ij\bar{k}}T_{\bar{i}\bar{j}a}
	-|\n T|^2-|\nb T|^2
\\
&=4\C_{b\bar{k}j\bar{a}}T_{ia\bar{b}}T_{\bar{i}\bar{j}k}
	+4T_{ib\bar{c}}T_{ja\bar{b}}T_{\bar{a}\bar{k}c}T_{\bar{i}\bar{j}k}
	-2Q_{a\bar{i}}T_{ij\bar{k}}T_{\bar{a}\bar{j}k}
	+Q_{a\bar{k}}T_{ij\bar{a}}T_{\bar{i}\bar{j}k}
	\q
	+2\n_{i}T_{ja\bar{b}}T_{\bar{k}\bar{a}b}T_{\bar{i}\bar{j}k}
	+2\n_{\bar{i}}T_{\bar{j}\bar{a}b}T_{ka\bar{b}}T_{ij\bar{k}}
	-|\n T|^2-|\nb T|^2
\end{align*}
\end{proof}

In dimension 2, we can rewrite the evolution equation in a more useful form.

\begin{lemma}\label{lem-torsionDim2}
In dimension 2, we have
\begin{align*}
\heat|T|^2
&=(\Ric+S)_{a\bar{b}}T_{ij\bar{a}}T_{\bar{i}\bar{j}b}
	+\n_{i}|T|^2T_{\bar{i}\bar{j}j}
	+\n_{\bar{i}}|T|^2T_{ij\bar{j}}
	-|T|^4-|\n T|^2-|\nb T|^2
\end{align*}
\end{lemma}
\begin{proof}
In dimension 2, we will calculate each term of the equation more clearly.
\begin{align*}
\heat|T|^2
&=4\C_{b\bar{k}j\bar{a}}T_{ia\bar{b}}T_{\bar{i}\bar{j}k}
	+4T_{ib\bar{c}}T_{ja\bar{b}}T_{\bar{a}\bar{k}c}T_{\bar{i}\bar{j}k}
	+\Big\{
		-2Q_{a\bar{i}}T_{ij\bar{k}}T_{\bar{a}\bar{j}k}
		+Q_{a\bar{k}}T_{ij\bar{a}}T_{\bar{i}\bar{j}k}
	\Big\}
	\q
	+2\Big\{
		\n_{i}T_{ja\bar{b}}T_{\bar{k}\bar{a}b}T_{\bar{i}\bar{j}k}
	 +\n_{\bar{i}}T_{\bar{j}\bar{a}b}T_{ka\bar{b}}T_{ij\bar{k}}
  \Big\}
	-|\n T|^2-|\nb T|^2
\end{align*}
The first term is
\begin{align*}
4\C_{b\bar{k}j\bar{a}}T_{ia\bar{b}}T_{\bar{i}\bar{j}k}
&=4\C_{b\bar{k}2\bar{2}}T_{12\bar{b}}T_{\bar{1}\bar{2}k}
	+4\C_{b\bar{k}1\bar{1}}T_{21\bar{b}}T_{\bar{2}\bar{1}k}
\\
&=4(\C_{b\bar{k}1\bar{1}}+\C_{b\bar{k}2\bar{2}})
		T_{12\bar{b}}T_{\bar{1}\bar{2}k}
\\
&=4\Ric_{b\bar{k}}T_{12\bar{b}}T_{\bar{1}\bar{2}k}
\\
&=4T_{12\bar{1}}T_{\bar{1}\bar{2}1}
	(\C_{1\bar{1}1\bar{1}}+\C_{1\bar{1}2\bar{2}})
	+4T_{12\bar{1}}T_{\bar{1}\bar{2}2}\Ric_{1\bar{2}}
	\q
	+4T_{12\bar{2}}T_{\bar{1}\bar{2}1}\Ric_{2\bar{1}}
	+4T_{12\bar{2}}T_{\bar{1}\bar{2}2}
	(\C_{2\bar{2}1\bar{1}}+\C_{2\bar{2}2\bar{2}})
\end{align*}
The second term is
\begin{align*}
4T_{ib\bar{c}}T_{ja\bar{b}}T_{\bar{a}\bar{k}c}T_{\bar{i}\bar{j}k}
&=4T_{12\bar{c}}T_{21\bar{2}}T_{\bar{1}\bar{2}c}T_{\bar{1}\bar{2}2}
	+4T_{21\bar{c}}T_{12\bar{1}}T_{\bar{2}\bar{1}c}T_{\bar{2}\bar{1}1}
\\
&=-4|T_{12\bar{c}}|^2(|T_{12\bar{1}}|^2+|T_{12\bar{2}}|^2)
\\
&=-|T|^4
\end{align*}
Terms in the brace in the first line are
\begin{align*}
&-2Q_{a\bar{i}}T_{ij\bar{k}}T_{\bar{a}\bar{j}k}
		+Q_{a\bar{k}}T_{ij\bar{a}}T_{\bar{i}\bar{j}k}
\\
=&-2Q_{1\bar{1}}T_{12\bar{k}}T_{\bar{1}\bar{2}k}
	-2Q_{2\bar{2}}T_{21\bar{k}}T_{\bar{2}\bar{1}k}
		+Q_{a\bar{k}}T_{12\bar{a}}T_{\bar{1}\bar{2}k}
		+Q_{a\bar{k}}T_{21\bar{a}}T_{\bar{2}\bar{1}k}
\\
=&-2|T_{12\bar{k}}|^2(Q_{1\bar{1}}+Q_{2\bar{2}})
	+2T_{ac\bar{d}}T_{\bar{k}\bar{c}d}T_{12\bar{a}}T_{\bar{1}\bar{2}k}
\\
=&-|T|^4
	+2T_{21\bar{d}}T_{\bar{2}\bar{1}d}T_{12\bar{2}}T_{\bar{1}\bar{2}2}
	+2T_{12\bar{d}}T_{\bar{1}\bar{2}d}T_{12\bar{1}}T_{\bar{1}\bar{2}1}
\\
=&-|T|^4+\frac{1}{2}|T|^4
\\
=&-\frac{1}{2}|T|^4
\end{align*}
Then we address the brace in the second line.
Notice that
\begin{align*}
&\n_{i}T_{ja\bar{b}}T_{\bar{k}\bar{a}b}T_{\bar{i}\bar{j}k}
\\
=&\n_{1}T_{21\bar{b}}T_{\bar{2}\bar{1}b}T_{\bar{1}\bar{2}2}
	+\n_{2}T_{12\bar{b}}T_{\bar{1}\bar{2}b}T_{\bar{2}\bar{1}1}
\\
=&\n_{1}(|T_{12\bar{b}}|^2)T_{\bar{1}\bar{2}2}
	-T_{12\bar{b}}\n_{1}T_{\bar{1}\bar{2}b}T_{\bar{1}\bar{2}2}
	+\n_{2}(|T_{12\bar{b}}|^2)T_{\bar{2}\bar{1}1}
	-T_{12\bar{b}}\n_{2}T_{\bar{1}\bar{2}b}T_{\bar{2}\bar{1}1}
\\
=&\frac{1}{2}(\n_{i}|T|^2T_{\bar{i}\bar{j}j})
	-T_{12\bar{b}}\n_{1}T_{\bar{1}\bar{2}b}T_{\bar{1}\bar{2}2}
	-T_{12\bar{b}}\n_{2}T_{\bar{1}\bar{2}b}T_{\bar{2}\bar{1}1}
\end{align*}
Using the first Bianchi identity, we have
\begin{align*}
&-T_{12\bar{b}}\n_{1}T_{\bar{1}\bar{2}b}T_{\bar{1}\bar{2}2}
	-T_{12\bar{b}}\n_{2}T_{\bar{1}\bar{2}b}T_{\bar{2}\bar{1}1}
\\
=&-T_{12\bar{b}}(\C_{1\bar{2}b\bar{1}}-\C_{1\bar{1}b\bar{2}})T_{\bar{1}\bar{2}2}
	-T_{12\bar{b}}(\C_{2\bar{2}b\bar{1}}-\C_{2\bar{1}b\bar{2}})T_{\bar{2}\bar{1}1}
\\
=&-T_{12\bar{1}}		 	
	 (\C_{1\bar{2}1\bar{1}}-\C_{1\bar{1}1\bar{2}})T_{\bar{1}\bar{2}2}
	-T_{12\bar{2}}
	 (\C_{1\bar{2}2\bar{1}}-\C_{1\bar{1}2\bar{2}})T_{\bar{1}\bar{2}2}
	\q
	-T_{12\bar{1}}
	 (\C_{2\bar{2}1\bar{1}}-\C_{2\bar{1}1\bar{2}})T_{\bar{2}\bar{1}1}
	-T_{12\bar{2}}
	 (\C_{2\bar{2}2\bar{1}}-\C_{2\bar{1}2\bar{2}})T_{\bar{2}\bar{1}1}	
\\
=&T_{12\bar{1}}T_{\bar{1}\bar{2}1}
	 (\C_{2\bar{2}1\bar{1}}-\C_{2\bar{1}1\bar{2}})
	 +T_{12\bar{1}}T_{\bar{1}\bar{2}2}
	 (\C_{1\bar{1}1\bar{2}}-\C_{1\bar{2}1\bar{1}})
	 \q
	 +T_{12\bar{2}}T_{\bar{1}\bar{2}1}
	 (\C_{2\bar{2}2\bar{1}}-\C_{2\bar{1}2\bar{2}})
		+T_{12\bar{2}}T_{\bar{1}\bar{2}2}
	 (\C_{1\bar{1}2\bar{2}}-\C_{1\bar{2}2\bar{1}})
\end{align*}
So we get
\begin{align*}
&2\Big\{
		\n_{i}T_{ja\bar{b}}T_{\bar{k}\bar{a}b}T_{\bar{i}\bar{j}k}
	 +\n_{\bar{i}}T_{\bar{j}\bar{a}b}T_{ka\bar{b}}T_{ij\bar{k}}
  \Big\}
\\
=&\n_{i}|T|^2T_{\bar{i}\bar{j}j}
	+\n_{\bar{i}}|T|^2T_{ij\bar{j}}
	\q
	+2T_{12\bar{1}}T_{\bar{1}\bar{2}1}
	 (\C_{2\bar{2}1\bar{1}}-\C_{2\bar{1}1\bar{2}}
	 	+\C_{2\bar{2}1\bar{1}}-\C_{1\bar{2}2\bar{1}})
	+2T_{12\bar{1}}T_{\bar{1}\bar{2}2}
	(\C_{1\bar{1}1\bar{2}}-\C_{1\bar{2}1\bar{1}}
	+\C_{2\bar{2}1\bar{2}}-\C_{1\bar{2}2\bar{2}})
	\q
	+2T_{12\bar{2}}T_{\bar{1}\bar{2}1}
	 (\C_{2\bar{2}2\bar{1}}-\C_{2\bar{1}2\bar{2}}
	 	+\C_{1\bar{1}2\bar{1}}-\C_{2\bar{1}1\bar{1}})
	+2T_{12\bar{2}}T_{\bar{1}\bar{2}2}
	 (\C_{1\bar{1}2\bar{2}}-\C_{1\bar{2}2\bar{1}}
	 +\C_{1\bar{1}2\bar{2}}-\C_{2\bar{1}1\bar{2}})
\\
=&\n_{i}|T|^2T_{\bar{i}\bar{j}j}
	+\n_{\bar{i}}|T|^2T_{ij\bar{j}}
	\q
	+2T_{12\bar{1}}T_{\bar{1}\bar{2}1}
	 (\C_{2\bar{2}1\bar{1}}
	 -\C_{1\bar{1}2\bar{2}}+T_{12\bar{c}}T_{\bar{1}\bar{2}c})
	+2T_{12\bar{1}}T_{\bar{1}\bar{2}2}
	(S_{1\bar{2}}-\Ric_{1\bar{2}})
	\q
	+2T_{12\bar{2}}T_{\bar{1}\bar{2}1}
	 (S_{2\bar{1}}-\Ric_{2\bar{1}})
	+2T_{12\bar{2}}T_{\bar{1}\bar{2}2}
	 (\C_{1\bar{1}2\bar{2}}
	 -\C_{2\bar{2}1\bar{1}}+T_{12\bar{c}}T_{\bar{1}\bar{2}c})
\\
=&\n_{i}|T|^2T_{\bar{i}\bar{j}j}
	+\n_{\bar{i}}|T|^2T_{ij\bar{j}}
	+\frac{1}{2}|T|^4
	\q
	+2T_{12\bar{1}}T_{\bar{1}\bar{2}1}
	 (\C_{2\bar{2}1\bar{1}}-\C_{1\bar{1}2\bar{2}})
	+2T_{12\bar{1}}T_{\bar{1}\bar{2}2}
	(S_{1\bar{2}}-\Ric_{1\bar{2}})
	\q
	+2T_{12\bar{2}}T_{\bar{1}\bar{2}1}
	 (S_{2\bar{1}}-\Ric_{2\bar{1}})
	+2T_{12\bar{2}}T_{\bar{1}\bar{2}2}
	 (\C_{1\bar{1}2\bar{2}}-\C_{2\bar{2}1\bar{1}})
\end{align*}
The last equality is because of the pluriclosed condition (Lemma \ref{lem-PCcondition}).
Combining all terms above together, we obtain
\begin{align*}
\heat|T|^2
&=\n_{i}|T|^2T_{\bar{i}\bar{j}j}
	+\n_{\bar{i}}|T|^2T_{ij\bar{j}}
	-|T|^4-|\n T|^2-|\nb T|^2
	\q
	+2T_{12\bar{1}}T_{\bar{1}\bar{2}1}
	 (\Ric_{1\bar{1}}+\C_{1\bar{1}1\bar{1}}+\C_{2\bar{2}1\bar{1}})
	+2T_{12\bar{1}}T_{\bar{1}\bar{2}2}
	(S_{1\bar{2}}+\Ric_{1\bar{2}})
	\q
	+2T_{12\bar{2}}T_{\bar{1}\bar{2}1}
	 (S_{2\bar{1}}+\Ric_{2\bar{1}})
	+2T_{12\bar{2}}T_{\bar{1}\bar{2}2}
	 (\Ric_{2\bar{2}}+\C_{1\bar{1}2\bar{2}}+\C_{2\bar{2}2\bar{2}})
\\
&=\n_{i}|T|^2T_{\bar{i}\bar{j}j}
	+\n_{\bar{i}}|T|^2T_{ij\bar{j}}
	-|T|^4-|\n T|^2-|\nb T|^2
	\q
	+2T_{12\bar{1}}T_{\bar{1}\bar{2}1}
	 (\Ric_{1\bar{1}}+S_{1\bar{1}})
	+2T_{12\bar{1}}T_{\bar{1}\bar{2}2}
	(\Ric_{1\bar{2}}+S_{1\bar{2}})
	\q
	+2T_{12\bar{2}}T_{\bar{1}\bar{2}1}
	 (\Ric_{2\bar{1}}+S_{2\bar{1}})
	+2T_{12\bar{2}}T_{\bar{1}\bar{2}2}
	 (\Ric_{2\bar{2}}+S_{2\bar{2}})
\\
&=\n_{i}|T|^2T_{\bar{i}\bar{j}j}
	+\n_{\bar{i}}|T|^2T_{ij\bar{j}}
	-|T|^4-|\n T|^2-|\nb T|^2
	\q
	+(\Ric+S)_{a\bar{b}}T_{ij\bar{a}}T_{\bar{i}\bar{j}b}
\end{align*}
\end{proof}

Now, we give the proof of Theorem \ref{thm-HSdim2-T2}.
\begin{proof}[Proof of Theorem \ref{thm-HSdim2-T2}.]
From Lemma \ref{lem-torsionDim2}, we know that there exists a universal constant $c$ such that
\begin{align*}
\heat|T|^2
&\leq cK|T|^2
	+\n_{i}|T|^2T_{\bar{i}\bar{j}j}
	+\n_{\bar{i}}|T|^2T_{ij\bar{j}}
	-|T|^4-|\n T|^2-|\nb T|^2
\end{align*}
Since the manifold $M\times[0,t]$ is compact for any time $t<\tau$, the function $|T|^2(x,t)$ can always achieve its maximum.
If $|T|^2(x,t)$ achieves the maximum at a parabolic interior point $(x_m,t_m)$, then
\begin{align*}
0\leq cK|T|^2(x_m,t_m)-|T|^4(x_m,t_m),
\end{align*}
which implies
\begin{align*}
|T|^2(x,t)\leq|T|^2(x_m,t_m)\leq cK.
\end{align*}
Otherwise, the maximum is achieved at the parabolic boundary.
This implies
\begin{align*}
|T|^2(x,t)\leq \max_{x\in M}|T|^2(x,0)
\end{align*}
So we conclude that
\begin{align*}
|T(x,t)|^2_{g(x,t)}\leq \max\left\{cK,\max_{x\in M}|T(x,0)|^2_{g(x,0)}\right\}.
\end{align*}
\end{proof}

\section{Hermitian-symplectic case}\label{sec-HSflow}
We begin with some facts about Hermitian-symplectic forms (briefly, HS forms).
One can find more details in \cite{2022arXiv221204060Y}.

On a complex manifold $(M^{2n},J)$, a HS form $H$  is a symplectic form, i.e. real and closed nondegenerate two form, with positive (1,1)-part.
We may write it by bi-degree as $H=\me+\p+\ov{\p}$, where $\me$ is the (1,1)-part and $\p$ is the (2,0)-part.
Since $H$ is closed, we have the relationship
\begin{align*}
\dd\p=0
,\quad
\dd\me+\db\p=0.
\end{align*}
Thus the (1,1)-part is a pluriclosed metric.
And for convenience, we say a pluriclosed metric is Hermitian-symplectic if it is the (1,1)-part of some HS forms.
Pluriclosed flow preserves the Hermitian-symplectic condition.
And we can extend it to a flow of HS forms by evolving $\p$ along the (2,0)-part of Bismut Ricci form.
\begin{equation*}
\left\{
\begin{array}{ll}
\dt\me=-\rho^{1,1}(\me)
\\
\dt\p=-\rho^{2,0}(\me)
\end{array}
\right.
\end{equation*}
Or equivalently,
\begin{align*}
\dt H=-\rho(H^{1,1})
\end{align*}
Like Ricci flow or K{\"a}hler Ricci flow, we can consider solitons of this flow.
\begin{definition}
A Hermitian-symplectic soliton is a HS form $H$ satisfying
\begin{align*}
\rho(H^{1,1})=\lambda H+\mathcal{L}_{X}H
\end{align*}
where $\lambda$ is a real constant and $X$ is a holomorphic vector field.
\end{definition}

\begin{remark}
We can write the HS soliton by bidegree as
\begin{align*}
\rho^{1,1}(\me)=\lambda\me+\mathcal{L}_{X}\me
,\quad
\rho^{2,0}(\me)=\lambda\p+\mathcal{L}_{X}\p.
\end{align*}
Thus HS solitons are special solitons of pluriclosed flow that have additional requirement on the (2,0)-part of Bismut Ricci form.
\end{remark}

%\subsection{Evolution equation and a monotonic quantity}
In this subsection, we will compute the evolution equation of torsion in Hermitian-symplectic case and give a monotontic quantity along pluriclosed flow with Hermitian-symplectic initial data.
As an application, we prove that all Hermitian-symplectic solitons are K{\"a}hler Ricci solitons. 

First of all, we state a lemma given by the authors of \cite{2021arXiv210613716G}.
\begin{lemma}[Proposition 3.24 of \cite{2021arXiv210613716G}]\label{lem-divTorsion}
For a pluriclosed metric, we have
\begin{align*}
\rho^{2,0}(\me)=-\text{\rm div}T
\end{align*}
where $(\text{\rm div}T)_{ij}=-\n_{a}T_{ij\bar{a}}$.
\end{lemma}

Then we can rewrite the evolution equation of $\p$.
\begin{proposition}\label{prop-ev-phi}
The evolution equation of $\p$ is equivalent to
\begin{align*}
\heat\p=0
\end{align*}
\end{proposition}
\begin{proof}
Assume $\p=\frac{1}{2}\p_{ij}dz^i\wedge dz^j$ satisfying $\p_{ij}+\p_{ji}=0$.
By direct computation, we have
\begin{align*}
\dd\me=\dd_{i}g_{j\bar{k}}dz^{i}\wedge dz^{j}\wedge d\bar{z}^{k}
,\quad
\db\p=\frac{1}{2}\dd_{\bar{k}}\p_{ij}dz^{i}\wedge dz^{j}\wedge d\bar{z}^{k}
\end{align*}
Since $\dd\me+\db\p=0$, we get
\begin{align}\label{eq-TorsionDbphi}
\n_{\bar{k}}\p_{ij}
&=-\dd_{i}g_{j\bar{k}}+\dd_{j}g_{i\bar{k}}
=-T_{ij\bar{k}}
\end{align}
From Lemme \ref{lem-divTorsion}, we obtain
\begin{align*}
(\rho^{2,0})_{ij}
&=\n_{a}T_{ij\bar{a}}
=-\n_{a}\n_{\bar{a}}\p_{ij}
=\D\p_{ij}
\end{align*}
So we can reformulate the evolution equation of $\p$ as
\begin{align*}
\heat\p=0
\end{align*}
\end{proof}

\begin{remark}
The Proposition 4.10 of \cite{2015arXiv150202584S} gives a similar result for a more general situation than Hermitian-symplectic case.
\end{remark}

The next proposition gives the monotonic quantity.

\begin{proposition}\label{thm-monotonicity}
Let $H(t)=\me(t)+\p(t)+\ov{\p}(t)$ be a solution to the pluriclosed flow with Hermitian-symplectic initial data.
Then either $q(t)=\max\limits_{x\in M}|\p|_{g(t)}(x,t)$ is strictly monotonically decreasing, or $\me(t)$ is a solution to K{\"a}hler Ricci flow.
\end{proposition}
\begin{proof}
We will use notations defined in subsection \ref{subs-notation}.
By direct computation, we have
\begin{align*}
\dt|\p|^2
&=(\dt\p,\p)+(\p,\dt\p)+(\p,(\dt g)[\p])
\\
&=(\dt\p,\p)+(\p,\dt\p)-(\p,S[\p])+(\p,Q[\p])
\end{align*}
And
\begin{align*}
\D|\p|^2
&=(\D\p,\p)+(\p,\Db\p)+|\n\p|^2+|\nb\p|^2
\\
&=(\D\p,\p)+(\p,\D\p)-(\p,S\ci\p)+|\n\p|^2+|\nb\p|^2
\end{align*}
From Proposition \ref{prop-ev-phi}, we know that $\p$ satisfies the heat equation.
So
\begin{align*}
\heat|\p|^2
&=-(\p,S[\p])+(\p,Q[\p])+(\p,S\ci\p)-|\n\p|^2-|\nb\p|^2
\end{align*}
Noticing that $\p\in T^{2,0}M$, we obtain $S[\p]=S\ci\p$ from Remark \ref{re-circAnd[]}.
Thus 
\begin{align}\label{eq-ev-normPhi}
\heat|\p|^2
&=-|\n\p|^2-|\nb\p|^2+(\p,Q[\p])
\end{align}
Since $Q$ is non-negative, we have $(\p,Q[\p])\leq0$ (see Remark \ref{re-circAnd[]} again).

By the strong maximum principle, we know that either $q(t)=\max\limits_{x\in M}|\p|^2(x,t)$ is strictly monotonically decreasing, or $|\p|^2$ is constant on some time interval $[t_0,t_1]$.
For the latter case,	we have
\begin{align*}
|\n\p|^2=|\nb\p|^2=0
,\quad 
(\p,Q[\p])=0
,\quad
t\in[t_0,t_1]
\end{align*}
This is because those three terms on the right hand side of equation \eqref{eq-ev-normPhi} is non-positive.
Then we get $T=0$, $t\in[t_0,t_1]$ from equation \eqref{eq-TorsionDbphi}.
By the uniqueness of parabolic systems (e.g. see section 4.4 of \cite{MR2265040}), we conclude that $\me(t)$ is a solution to K{\"a}hler Ricci flow.
And we complete this proof.
\end{proof}

As an application, we show that all Hermitian-symplectic solitons are K{\"a}hler Ricci solitons.

\begin{corollary}\label{coro-HSS=KRS}
The (1,1)-part of Hermitian-symplectic solitons are K{\"a}hler Ricci solitons.
\end{corollary}
\begin{proof}
Let $H_0=\me_0+\p_0+\ov{\p}_0$ be a Hermitian-symplectic soliton satisfying
\begin{align*}
\rho(H_0^{1,1})=\lambda H_0+\mathcal{L}_{X}H_0
\end{align*}
where $\lambda\in\mathbb{R}$ and $X$ is a holomorphic vector field.
Assume $f_t$ is the flow generated by the time-dependent vector field $(\lambda t-1)^{-1}X$.
And we claim that $H(t)=(1-\lambda t)f^{*}_{t}(H_0)$ is a solution to pluriclosed flow with initial data $H_0$.
To see this,
\begin{align*}
\dt H(t)
&=	-\lambda f^{*}_{t}(H_0)
	+(1-\lambda t)f^{*}_{t}(\mathcal{L}_{(\lambda t-1)^{-1}X}H_0)
\\
&=-f^{*}_{t}(\lambda H_0+\mathcal{L}_{X}H_0)
\\
&=-f^{*}_{t}(\rho(H_0^{1,1}))
\\
&=-\rho(f^{*}_{t}(H_0)^{1,1})
\\
&=-\rho(H(t)^{1,1})
\end{align*}
For any positive constant $a$, we have
\begin{align*}
|a\p|_{a\me}(x,t)=|\p|_{\me}(x,t).
\end{align*}
Since $|\p|(x,t)$ is invariant under rescaling, we obtain that
\begin{align*}
\max_{x\in M}|\p|_{g(t)}(x,t)
=\max_{x\in M}|\p|_{g(0)}(x,0)
\end{align*}
for the solution $H(t)$.
So we conclude that $\me(t)$ is a solution to K{\"a}hler Ricci flow by Proposition \ref{thm-monotonicity}.
In particular, we obtain that $\me_0$ is a K{\"a}hler Ricci soliton.
\end{proof}

\section{Appendix}
In this appendix, we show the detail of the calculation of evolution equations.
\begin{proof}[Proof of Lemma \ref{lem-EvCurvature}.]
Using Lemma \ref{lem-EvChristoffel}, we have
\begin{align*}
\dt\C_{i\bar{j}p\bar{q}}
&=-\dt(\dd_{\bar{j}}\Gamma_{ip}^{a}g_{a\bar{q}})
=-\dd_{\bar{j}}\dt\Gamma_{ip}^{a}g_{a\bar{q}}
	-\dd_{\bar{j}}\Gamma_{ip}^{a}\dt g_{a\bar{q}}
\\
&=\n_{\bar{j}}\n_{i}S_{p\bar{q}}-\n_{\bar{j}}\n_{i}Q_{p\bar{q}}
	+\C_{i\bar{j}p\bar{a}}(-S_{a\bar{q}}+Q_{a\bar{q}})
\end{align*}

{\bf Step 1.} Compute $\n_{\bar{j}}\n_{i}S_{p\bar{q}}$.
Applying the second Bianchi identity, we get
\begin{align*}
\n_{\bar{j}}\n_{i}S_{p\bar{q}}
&=\n_{\bar{j}}\n_{i}\C_{a\bar{a}p\bar{q}}
=\n_{\bar{j}}(\n_{a}\C_{i\bar{a}p\bar{q}}
	+T_{ai}^{b}\C_{b\bar{a}p\bar{q}})
\\
&=\n_{a}\n_{\bar{j}}\C_{i\bar{a}p\bar{q}}
	+\C_{a\bar{j}i\bar{b}}\C_{b\bar{a}p\bar{q}}
	-\C_{a\bar{j}b\bar{a}}\C_{i\bar{b}p\bar{q}}
	+\C_{a\bar{j}p\bar{b}}\C_{i\bar{a}b\bar{q}}
	-\C_{a\bar{j}b\bar{q}}\C_{i\bar{a}p\bar{b}}
	\q
	+\n_{\bar{j}}T_{ai}^{b}\C_{b\bar{a}p\bar{q}}
	+T_{ai}^{b}\n_{\bar{j}}\C_{b\bar{a}p\bar{q}}
\end{align*}
Using the second Bianchi identity again, we have
\begin{align*}
\n_{a}\n_{\bar{j}}\C_{i\bar{a}p\bar{q}}
&=\n_{a}(\n_{\bar{a}}\C_{i\bar{j}p\bar{q}}
	+T_{\bar{a}\bar{j}}^{\bar{b}}\C_{i\bar{b}p\bar{q}})
\\
&=\n_{a}\n_{\bar{a}}\C_{i\bar{j}p\bar{q}}
	+\n_{a}T_{\bar{a}\bar{j}}^{\bar{b}}\C_{i\bar{b}p\bar{q}}
	+T_{\bar{a}\bar{j}}^{\bar{b}}\n_{a}\C_{i\bar{b}p\bar{q}}
\end{align*}
Recall the first Bianchi identity 
$\n_{\bar{j}}T_{ip\bar{q}}=\C_{p\bar{j}i\bar{q}}-\C_{i\bar{j}p\bar{q}}$.
Then we get
\begin{align*}
\n_{\bar{j}}\n_{i}S_{p\bar{q}}
&=\n_{a}\n_{\bar{a}}\C_{i\bar{j}p\bar{q}}
	+T_{ai}^{b}\n_{\bar{j}}\C_{b\bar{a}p\bar{q}}	
	+T_{\bar{a}\bar{j}}^{\bar{b}}\n_{a}\C_{i\bar{b}p\bar{q}}	
	\q
	+\C_{a\bar{j}i\bar{b}}\C_{b\bar{a}p\bar{q}}
	-\C_{a\bar{j}b\bar{a}}\C_{i\bar{b}p\bar{q}}
	+\C_{a\bar{j}p\bar{b}}\C_{i\bar{a}b\bar{q}}
	-\C_{a\bar{j}b\bar{q}}\C_{i\bar{a}p\bar{b}}
	\q
	+\C_{i\bar{j}a\bar{b}}\C_{b\bar{a}p\bar{q}}
	-\C_{a\bar{j}i\bar{b}}\C_{b\bar{a}p\bar{q}}
	+\C_{a\bar{j}b\bar{a}}\C_{i\bar{b}p\bar{q}}
	-\C_{a\bar{a}b\bar{j}}\C_{i\bar{b}p\bar{q}}
\end{align*}
{\bf Step 2.} Compute $\n_{\bar{j}}\n_{i}Q_{p\bar{q}}$.
\begin{align*}
\n_{\bar{j}}\n_{i}Q_{p\bar{q}}
&=\n_{\bar{j}}\n_{i}(T_{pa\bar{b}}T_{\bar{q}\bar{a}b})
\\
&=\n_{\bar{j}}\n_{i}T_{pa\bar{b}}T_{\bar{q}\bar{a}b}
	+T_{pa\bar{b}}\n_{\bar{j}}\n_{i}T_{\bar{q}\bar{a}b}
	+\n_{\bar{j}}T_{pa\bar{b}}\n_{i}T_{\bar{q}\bar{a}b}
	+\n_{i}T_{pa\bar{b}}\n_{\bar{j}}T_{\bar{q}\bar{a}b}
\end{align*}
The first term is
\begin{align*}
\n_{\bar{j}}\n_{i}T_{pa\bar{b}}T_{\bar{q}\bar{a}b}
&=\n_{i}\n_{\bar{j}}T_{pa\bar{b}}T_{\bar{q}\bar{a}b}
	+\C_{i\bar{j}p\bar{c}}T_{ca\bar{b}}T_{\bar{q}\bar{a}b}
	+\C_{i\bar{j}a\bar{c}}T_{pc\bar{b}}T_{\bar{q}\bar{a}b}
	-\C_{i\bar{j}c\bar{b}}T_{pa\bar{c}}T_{\bar{q}\bar{a}b}
\\
&=\n_{i}\C_{a\bar{j}p\bar{b}}T_{\bar{q}\bar{a}b}
	-\n_{i}\C_{p\bar{j}a\bar{b}}T_{\bar{q}\bar{a}b}
	\q
	+\C_{i\bar{j}p\bar{c}}T_{ca\bar{b}}T_{\bar{q}\bar{a}b}
	+\C_{i\bar{j}a\bar{c}}T_{pc\bar{b}}T_{\bar{q}\bar{a}b}
	-\C_{i\bar{j}c\bar{b}}T_{pa\bar{c}}T_{\bar{q}\bar{a}b}
\end{align*}
Thus from Bianchi identity, we have
\begin{align*}
\n_{\bar{j}}\n_{i}Q_{p\bar{q}}
&=\n_{i}\C_{a\bar{j}p\bar{b}}T_{\bar{q}\bar{a}b}
	-\n_{i}\C_{p\bar{j}a\bar{b}}T_{\bar{q}\bar{a}b}
	+\n_{\bar{j}}\C_{i\bar{a}b\bar{q}}T_{pa\bar{b}}
	-\n_{\bar{j}}\C_{i\bar{q}b\bar{a}}T_{pa\bar{b}}
	\q	
	+\C_{a\bar{j}p\bar{b}}\C_{i\bar{a}b\bar{q}}
	-\C_{p\bar{j}a\bar{b}}\C_{i\bar{a}b\bar{q}}
	-\C_{a\bar{j}p\bar{b}}\C_{i\bar{q}b\bar{a}}
	+\C_{p\bar{j}a\bar{b}}\C_{i\bar{q}b\bar{a}}
	\q
	+\C_{i\bar{j}p\bar{c}}T_{ca\bar{b}}T_{\bar{q}\bar{a}b}
	+\C_{i\bar{j}a\bar{c}}T_{pc\bar{b}}T_{\bar{q}\bar{a}b}
	-\C_{i\bar{j}c\bar{b}}T_{pa\bar{c}}T_{\bar{q}\bar{a}b}
	+\n_{i}T_{pa\bar{b}}\n_{\bar{j}}T_{\bar{q}\bar{a}b}
\end{align*}
{\bf Step 3.} Combining above together, we obtain
\begin{align*}
\heat\C_{i\bar{j}p\bar{q}}
&=\n_{a}\C_{i\bar{b}p\bar{q}}	T_{\bar{a}\bar{j}b}
	-\n_{i}\C_{a\bar{j}p\bar{b}}T_{\bar{q}\bar{a}b}
	+\n_{i}\C_{p\bar{j}a\bar{b}}T_{\bar{q}\bar{a}b}
	+\n_{\bar{j}}\C_{b\bar{a}p\bar{q}}T_{ai\bar{b}}	
	-\n_{\bar{j}}\C_{i\bar{a}b\bar{q}}T_{pa\bar{b}}
	\q
	+\n_{\bar{j}}\C_{i\bar{q}b\bar{a}}T_{pa\bar{b}}
	+\C_{a\bar{j}i\bar{b}}\C_{b\bar{a}p\bar{q}}
	-\C_{a\bar{j}b\bar{a}}\C_{i\bar{b}p\bar{q}}
	+\C_{a\bar{j}p\bar{b}}\C_{i\bar{a}b\bar{q}}
	-\C_{a\bar{j}b\bar{q}}\C_{i\bar{a}p\bar{b}}
	\q
	+\C_{i\bar{j}a\bar{b}}\C_{b\bar{a}p\bar{q}}
	-\C_{a\bar{j}i\bar{b}}\C_{b\bar{a}p\bar{q}}
	+\C_{a\bar{j}b\bar{a}}\C_{i\bar{b}p\bar{q}}
	-\C_{a\bar{a}b\bar{j}}\C_{i\bar{b}p\bar{q}}
	-\C_{a\bar{j}p\bar{b}}\C_{i\bar{a}b\bar{q}}
	\q
	+\C_{p\bar{j}a\bar{b}}\C_{i\bar{a}b\bar{q}}
	+\C_{a\bar{j}p\bar{b}}\C_{i\bar{q}b\bar{a}}
	-\C_{p\bar{j}a\bar{b}}\C_{i\bar{q}b\bar{a}}	
	-\C_{i\bar{j}p\bar{a}}S_{a\bar{q}}
	-\C_{i\bar{j}p\bar{c}}T_{ca\bar{b}}T_{\bar{q}\bar{a}b}
	\q
	-\C_{i\bar{j}a\bar{c}}T_{pc\bar{b}}T_{\bar{q}\bar{a}b}
	+\C_{i\bar{j}c\bar{b}}T_{pa\bar{c}}T_{\bar{q}\bar{a}b}
	+\C_{i\bar{j}p\bar{a}}Q_{a\bar{q}}
	-\n_{i}T_{pa\bar{b}}\n_{\bar{j}}T_{\bar{q}\bar{a}b}
\end{align*}
Briefly,
\begin{align*}
\heat\C
&=\n\C*\Tb+\nb\C*T+\C*\C+\C*T*\Tb+\n T*\ov{\n T}
\end{align*}
\end{proof}

\begin{proof}[Proof of Lemma \ref{lem-EvTorsion}.]
Using Lemma \ref{lem-EvChristoffel}, we have
\begin{align*}
\dt T_{ij\bar{k}}
&=\dt (\Ga_{ij}^{a}g_{a\bar{k}}-\Ga_{ji}^{a}g_{a\bar{k}})
=\dt(\Ga_{ij}^{a}-\Ga_{ji}^{a})g_{a\bar{k}}
	+(\Ga_{ij}^{a}-\Ga_{ji}^{a})\dt g_{a\bar{k}}
\\
&=-\n_{i}S_{j\bar{k}}+\n_{j}S_{i\bar{k}}
	+\n_{i}Q_{j\bar{k}}-\n_{j}Q_{i\bar{k}}
	+T_{ij\bar{a}}(-S_{a\bar{k}}+Q_{a\bar{k}})
\end{align*}
{\bf Step 1.} Compute $-\n_{i}S_{j\bar{k}}+\n_{j}S_{i\bar{k}}$.
Using Bianchi identity, we have
\begin{align*}
-\n_{i}S_{j\bar{k}}+\n_{j}S_{i\bar{k}}
&=-\n_{i}\C_{a\bar{a}j\bar{k}}+\n_{j}\C_{a\bar{a}i\bar{k}}
\\
&=-\n_{a}\C_{i\bar{a}j\bar{k}}-T_{ai}^{b}\C_{b\bar{a}j\bar{k}}
	+\n_{a}\C_{j\bar{a}i\bar{k}}+T_{aj}^{b}\C_{b\bar{a}i\bar{k}}
\\
&=\n_{a}(\C_{j\bar{a}i\bar{k}}-\C_{i\bar{a}j\bar{k}})
	-T_{ai}^{b}\C_{b\bar{a}j\bar{k}}
	+T_{aj}^{b}\C_{b\bar{a}i\bar{k}}
\\
&=\n_{a}\n_{\bar{a}}T_{ij\bar{k}}
	-T_{ai}^{b}\C_{b\bar{a}j\bar{k}}
	+T_{aj}^{b}\C_{b\bar{a}i\bar{k}}
\end{align*}
{\bf Step 2.} Compute $\n_{i}Q_{j\bar{k}}-\n_{j}Q_{i\bar{k}}$.
\begin{align*}
\n_{i}Q_{j\bar{k}}
&=\n_{i}(T_{ja\bar{b}}T_{\bar{k}\bar{a}b})
=\n_{i}T_{ja\bar{b}}T_{\bar{k}\bar{a}b}
	+T_{ja\bar{b}}\n_{i}T_{\bar{k}\bar{a}b}
\\
&=\n_{i}T_{ja\bar{b}}T_{\bar{k}\bar{a}b}
	+\C_{i\bar{a}b\bar{k}}T_{ja\bar{b}}
	-\C_{i\bar{k}b\bar{a}}T_{ja\bar{b}}
\end{align*}
So we have
\begin{align*}
\n_{i}Q_{j\bar{k}}-\n_{j}Q_{i\bar{k}}
&=\C_{i\bar{a}b\bar{k}}T_{ja\bar{b}}
	-\C_{i\bar{k}b\bar{a}}T_{ja\bar{b}}
	-\C_{j\bar{a}b\bar{k}}T_{ia\bar{b}}
	+\C_{j\bar{k}b\bar{a}}T_{ia\bar{b}}
	\q
	+\n_{i}T_{ja\bar{b}}T_{\bar{k}\bar{a}b}
	-\n_{j}T_{ia\bar{b}}T_{\bar{k}\bar{a}b}
\end{align*}
{\bf Step 3.} Combining above together, we obtain
\begin{align*}
\heat T_{ij\bar{k}}
&=\C_{i\bar{a}b\bar{k}}T_{ja\bar{b}}
	-\C_{i\bar{k}b\bar{a}}T_{ja\bar{b}}
	-\C_{j\bar{a}b\bar{k}}T_{ia\bar{b}}
	+\C_{j\bar{k}b\bar{a}}T_{ia\bar{b}}
	+\C_{b\bar{a}i\bar{k}}T_{aj\bar{b}}
	-\C_{b\bar{a}j\bar{k}}T_{ai\bar{b}}
	\q
	-S_{a\bar{k}}T_{ij\bar{a}}
	+\n_{i}T_{ja\bar{b}}T_{\bar{k}\bar{a}b}
	-\n_{j}T_{ia\bar{b}}T_{\bar{k}\bar{a}b}
	+T_{ij\bar{a}}Q_{a\bar{k}}
\\
&=(\C_{i\bar{a}b\bar{k}}
		-\C_{i\bar{k}b\bar{a}}-\C_{b\bar{a}i\bar{k}})T_{ja\bar{b}}
	+(\C_{j\bar{k}b\bar{a}}
		+\C_{b\bar{a}j\bar{k}}-\C_{j\bar{a}b\bar{k}})T_{ia\bar{b}}
	\q
	-S_{a\bar{k}}T_{ij\bar{a}}
	+\n_{i}T_{ja\bar{b}}T_{\bar{k}\bar{a}b}
	-\n_{j}T_{ia\bar{b}}T_{\bar{k}\bar{a}b}
	+T_{ij\bar{a}}Q_{a\bar{k}}
\\
&=(-\C_{b\bar{k}i\bar{a}}
	+T_{ib\bar{c}}T_{\bar{a}\bar{k}c})T_{ja\bar{b}}
	+(\C_{b\bar{k}j\bar{a}}
	+T_{jb\bar{c}}T_{\bar{k}\bar{a}c})T_{ia\bar{b}}
	\q
	-S_{a\bar{k}}T_{ij\bar{a}}
	+\n_{i}T_{ja\bar{b}}T_{\bar{k}\bar{a}b}
	-\n_{j}T_{ia\bar{b}}T_{\bar{k}\bar{a}b}
	+T_{ij\bar{a}}Q_{a\bar{k}}
\end{align*}
We use the pluriclosed condition (Lemma \ref{lem-PCcondition}) in the last equality.
So we have
\begin{align*}
\heat T_{ij\bar{k}}
&=-\C_{b\bar{k}i\bar{a}}T_{ja\bar{b}}
	+\C_{b\bar{k}j\bar{a}}T_{ia\bar{b}}
	+T_{ib\bar{c}}T_{\bar{a}\bar{k}c}T_{ja\bar{b}}
	-T_{jb\bar{c}}T_{\bar{a}\bar{k}c}T_{ia\bar{b}}
	\q
	-S_{a\bar{k}}T_{ij\bar{a}}
	+\n_{i}T_{ja\bar{b}}T_{\bar{k}\bar{a}b}
	-\n_{j}T_{ia\bar{b}}T_{\bar{k}\bar{a}b}
	+T_{ij\bar{a}}Q_{a\bar{k}}
\end{align*}
%Briefly,
%\begin{align*}
%\heat T
%&=\C*T+\n T*\Tb+T*T*\Tb
%\end{align*}
\end{proof}

\begin{proof}[Proof of Lemma \ref{lem-Ev-higherDerivativeCurvature}.]
Our proof is by induction on $m$.
From Lemma \ref{lem-EvCurvature}, we know that the evolution equation of curvature is
\begin{align*}
\heat\C
&=D\C*T+\C^{*2}+\C*T^{*2}+\n T*\ov{\n T}
\end{align*}
So the case of $m=0$ has been checked.

Combining Lemma \ref{lem-exdtAndconnection} and Lemma \ref{lem-ex-LaplaceAndconnection}, we get
\begin{align}\label{eq-ex-heatD}
[\heat,D]A=(D\C+DT*T)*A+\C*DA+T*D^2A
\end{align}
By induction hypothesis, we have
\begin{align*}
\heat D^{m+1}\C
&=D\heat D^{m}\C
	+(D\C+DT*T)*D^{m}\C+\C*D^{m+1}\C+T*D^{m+2}\C
\\
&=D(\sum_{k=1}^{m+1}D^{k}\C*D^{m+1-k}T)
	+T*D^{m+2}\C
	\q
	+D^{m+1}\left(\C^{*2}+\C*T^{*2}\right)
	+(D\C+DT*T)*D^{m}\C+\C*D^{m+1}\C
	\q
	+D(\sum_{k=1}^{m+1}\n^{k}T*\nb^{m+2-k}\Tb)
\\
&=\sum_{k=1}^{m+2}D^{k}\C*D^{m+2-k}T
	+D^{m+1}\left(\C^{*2}+\C*T^{*2}\right)
	+D(\sum_{k=1}^{m+1}\n^{k}T*\nb^{m+2-k}\Tb)
\end{align*}
Next we address the last term.
\begin{align*}
D(\sum_{k=1}^{m+1}\n^{k}T*\nb^{m+2-k}\Tb)
&=\sum_{k=1}^{m+1}\n^{k+1}T*\nb^{m+2-k}\Tb
	+\sum_{k=1}^{m+1}\n^{k}T*\nb^{m+3-k}\Tb
	\q
	+\sum_{k=1}^{m+1}\nb\n^{k}T*\nb^{m+2-k}\Tb
	+\sum_{k=1}^{m+1}\n^{k}T*\n\nb^{m+2-k}\Tb
\\
&=\sum_{k=1}^{m+2}\n^{k}T*\nb^{m+3-k}\Tb
	+\Big\{\sum_{k=1}^{m+1}(\n^{k}\C+\n^{k-1}(\C*T))*\nb^{m+2-k}\Tb
	\q
	+\sum_{k=1}^{m+1}\n^{k}T*\ov{(\n^{m+2-k}\C+\n^{m+1-k}(\C*T))}
	\Big\}
\end{align*}
The second equal sign is because of equation \eqref{eq-nbnmTorsion}.
Combining above together and noticing that terms in the curly bracket are included in $D^{m+1}\left(\C^{*2}+\C*T^{*2}\right)$, we obtain
\begin{align*}
\heat D^{m+1}\C
&=\sum_{k=1}^{m+2}D^{k}\C*D^{m+2-k}T
	+D^{m+1}\left(\C^{*2}+\C*T^{*2}\right)
	+\sum_{k=1}^{m+2}\n^{k}T*\nb^{m+3-k}\Tb
\end{align*}
And we complete our proof.
\end{proof}

\begin{proof}[Proof of Lemma \ref{lem-ev-highOrderDerivativeTorsion}.]
Our proof is by induction on $m$.
Recall the evolution equation of torsion (Lemma \ref{lem-EvTorsion})
\begin{align*}
\heat T
&=\C*T+\n T*\Tb+T*T*\Tb
\end{align*}
So the case of $m=0$ has been proved.
Combining Lemma \ref{lem-exdtAndconnection} and Lemma \ref{lem-ex-LaplaceAndconnection}, we obtain
\begin{align*}
[\heat,\n]A=(\n\C+\C*T+\n T*\Tb)*A+\C*\n A+T*\n\nb T
\end{align*}
Using the induction hypothesis, we get
\begin{align*}
\heat\n^{m+1}T
&=\n\heat\n^{m}T+\C*\n^{m+1}T+T*\n\nb \n^{m}T
	\q
	+(\n\C+\C*T+\n T*\Tb)*\n^{m}T
\\
&=\n^{m+1}(\C*T+T^{*2}*\Tb)+\n^{m+2}(T*\Tb)
	+\Big\{T*\n^{m+1}\C+T*\n^{m}(\C*T)
	\q
	+(\n\C+\C*T+\n T*\Tb)*\n^{m}T
	\Big\}
\\
&=\n^{m+1}(\C*T+T^{*2}*\Tb)+\n^{m+2}(T*\Tb)
\end{align*}
The third line is because of formula \eqref{eq-nbnmTorsion}.
And terms in the curly bracket are included in the first term in the last line.
This is because of
\begin{align*}
\n^{m+1}(T^{*2}*\Tb)
&=\n^{m}(T^{*2}*\C)+\n^{m}(\n T*T*\Tb)
\end{align*}
So we complete our proof.
\end{proof}

\bibliographystyle{plain}
\bibliography{tehss}

\begin{thebibliography}{10}

\bibitem{MR2030225}
Wolf~P. Barth, Klaus Hulek, Chris A.~M. Peters, and Antonius Van~de Ven.
\newblock {\em Compact complex surfaces}, volume~4 of {\em Ergebnisse der
  Mathematik und ihrer Grenzgebiete. 3. Folge. A Series of Modern Surveys in
  Mathematics [Results in Mathematics and Related Areas. 3rd Series. A Series
  of Modern Surveys in Mathematics]}.
\newblock Springer-Verlag, Berlin, second edition, 2004.

\bibitem{MR2061425}
Bennett Chow and Dan Knopf.
\newblock {\em The {R}icci flow: an introduction}, volume 110 of {\em
  Mathematical Surveys and Monographs}.
\newblock American Mathematical Society, Providence, RI, 2004.

\bibitem{MR2274812}
Bennett Chow, Peng Lu, and Lei Ni.
\newblock {\em Hamilton's {R}icci flow}, volume~77 of {\em Graduate Studies in
  Mathematics}.
\newblock American Mathematical Society, Providence, RI; Science Press Beijing,
  New York, 2006.

\bibitem{MR3319954}
Nicola Enrietti, Anna Fino, and Luigi Vezzoni.
\newblock The pluriclosed flow on nilmanifolds and tamed symplectic forms.
\newblock {\em J. Geom. Anal.}, 25(2):883--909, 2015.

\bibitem{2021arXiv210613716G}
Mario {Garcia-Fernandez}, Joshua {Jordan}, and Jeffrey {Streets}.
\newblock {Non-K{\"a}hler Calabi-Yau geometry and pluriclosed flow}.
\newblock {\em arXiv e-prints}, page arXiv:2106.13716, June 2021.

\bibitem{2002math.....11159P}
Grisha {Perelman}.
\newblock {The entropy formula for the Ricci flow and its geometric
  applications}.
\newblock {\em arXiv Mathematics e-prints}, page math/0211159, November 2002.

\bibitem{2003math......7245P}
Grisha {Perelman}.
\newblock {Finite extinction time for the solutions to the Ricci flow on
  certain three-manifolds}.
\newblock {\em arXiv Mathematics e-prints}, page math/0307245, July 2003.

\bibitem{2003math......3109P}
Grisha {Perelman}.
\newblock {Ricci flow with surgery on three-manifolds}.
\newblock {\em arXiv Mathematics e-prints}, page math/0303109, March 2003.

\bibitem{2015arXiv150202584S}
Jeffrey {Streets}.
\newblock {Pluriclosed flow, Born-Infeld geometry, and rigidity results for
  generalized K{\"a}hler manifolds}.
\newblock {\em arXiv e-prints}, page arXiv:1502.02584, February 2015.

\bibitem{MR4023384}
Jeffrey Streets.
\newblock Classification of solitons for pluriclosed flow on complex surfaces.
\newblock {\em Math. Ann.}, 375(3-4):1555--1595, 2019.

\bibitem{MR4181011}
Jeffrey Streets.
\newblock Pluriclosed flow and the geometrization of complex surfaces.
\newblock In {\em Geometric analysis---in honor of {G}ang {T}ian's 60th
  birthday}, volume 333 of {\em Progr. Math.}, pages 471--510.
  Birkh\"{a}user/Springer, Cham, [2020] \copyright 2020.

\bibitem{MR2673720}
Jeffrey Streets and Gang Tian.
\newblock A parabolic flow of pluriclosed metrics.
\newblock {\em Int. Math. Res. Not. IMRN}, (16):3101--3133, 2010.

\bibitem{MR2781927}
Jeffrey Streets and Gang Tian.
\newblock Hermitian curvature flow.
\newblock {\em J. Eur. Math. Soc. (JEMS)}, 13(3):601--634, 2011.

\bibitem{MR2755684}
Jeffrey Streets and Gang Tian.
\newblock Regularity theory for pluriclosed flow.
\newblock {\em C. R. Math. Acad. Sci. Paris}, 349(1-2):1--4, 2011.

\bibitem{MR4287690}
Jeffrey Streets and Yury Ustinovskiy.
\newblock Classification of generalized {K}\"{a}hler-{R}icci solitons on
  complex surfaces.
\newblock {\em Comm. Pure Appl. Math.}, 74(9):1896--1914, 2021.

\bibitem{MR2265040}
Peter Topping.
\newblock {\em Lectures on the {R}icci flow}, volume 325 of {\em London
  Mathematical Society Lecture Note Series}.
\newblock Cambridge University Press, Cambridge, 2006.

\bibitem{2022arXiv221204060Y}
Yanan {Ye}.
\newblock {Bismut Einstein metrics on compact complex manifolds}.
\newblock {\em arXiv e-prints}, page arXiv:2212.04060, December 2022.

\bibitem{2022arXiv220712643Y}
Yanan {Ye}.
\newblock {Pluriclosed flow and Hermitian-symplectic structures}.
\newblock {\em arXiv e-prints}, page arXiv:2207.12643, July 2022.

\end{thebibliography}
\end{document}